\documentclass[a4paper,reqno,12pt]{amsart}
\usepackage{rotating,pdflscape,color,amssymb,hyperref,booktabs,subfigure,psfrag,amsmath,eufrak,bbm,multirow}
\usepackage[colorinlistoftodos]{todonotes}
\author[Sheehan Olver]{Sheehan Olver}\address{Sheehan Olver, School of Mathematics and Statistics, The University of Sydney, Australia.}
\email{Sheehan.Olver@sydney.edu.au}
\urladdr{http://www.maths.usyd.edu.au/u/olver/}

\author[Raj Rao Nadakuditi]{Raj Rao Nadakuditi}\address{Raj Rao Nadakuditi, Department of Electrical Engineering and Computer Science, University of Michigan, 1301 Beal Avenue, Ann Arbor, MI 48109. USA.}
\email{rajnrao@eecs.umich.edu}
\urladdr{http://www.eecs.umich.edu/\~\/rajnrao/}

\title[Free probability calculator]{Numerical computation of convolutions in free probability theory}

\date{\today}

\newcommand{\bxp}{\boxplus}
\newcommand{\bxt}{\boxtimes}
\newcommand{\bxd}{\boxdot}

\newtheorem{Th}{Theorem}[section]

\newtheorem{propo}[Th]{Proposition}

\newtheorem{lem}[Th]{Lemma}

\newtheorem{cor}[Th]{Corollary}

\long\def\symbolfootnote[#1]#2{\begingroup
\def\thefootnote{\fnsymbol{footnote}}\footnote[#1]{#2}\endgroup}

\theoremstyle{definition}
\newtheorem{Def}[Th]{Definition}

\newcommand{\ud}{\mathrm{d}}
\newcommand{\pro}{probability }

\newcommand{\f}{\dfrac}

\newcommand{\supp}{\operatorname{supp}}

\def\C{{\mathbb C}}

\definecolor{DarkGreen}{rgb}{0,.55,0}

\newcount\subitemcount
\newcount\itemcount
\theoremstyle{definition}
\newtheorem{remark}{Remark}
\newtheorem{algorithm}{Algorithm}
\def\algref#1{Algorithm~\ref{alg:#1}}
\def\Algorithm#1#2\par{
		\def\given{\vskip .1in\par Given }
		\def\return##1;{ compute ##1 as follows:}		
		\def\step{\vskip .1in \par {\bf \the\itemcount: \quad} \advance\itemcount by 1 \subitemcount=1}
		\def\substep{\vskip .1in\par$\qquad${\tt \the\subitemcount: \quad} \advance\subitemcount by 1}	
	\begin{algorithm}\label{alg:#1}
		\itemcount=1		
		#2
	\end{algorithm}
}

\newcommand{\smallbox}[1]{\mbox{\small$\displaystyle #1$}}
\newcommand{\tinybox}[1]{\mbox{\tiny$\displaystyle #1$}}

\def\sgn{{\rm sgn}\,}

\begin{document}

\begin{abstract}
We develop a numerical approach for computing the additive, multiplicative and compressive convolution operations from free probability theory.  We utilize the regularity properties of free convolution to identify (pairs of) `admissible' measures whose convolution results in a so-called `invertible measure' which is either a smoothly-decaying measure supported on the entire real line (such as the Gaussian) or square-root decaying measure supported on a compact interval (such as the semi-circle). This class of measures is important because these measures along with their Cauchy transforms can be accurately represented via a Fourier or Chebyshev series expansion, respectively.  Thus, knowledge of the functional inverse of their Cauchy transform suffices for numerically recovering the invertible measure via a non-standard yet well-behaved Vandermonde system of equations. We describe explicit algorithms for computing the  inverse Cauchy transform alluded to and recovering the associated measure with spectral accuracy.  Convergence is guaranteed under broad assumptions on the input measures.  
\end{abstract}

\maketitle
\section{Introduction}

\newcommand{\minitab}[2][l]{\begin{tabular}{#1}#2\end{tabular}}
\setlength{\heavyrulewidth}{0.1em}
\newcommand{\otoprule}{\midrule[\heavyrulewidth]}
\renewcommand*\arraystretch{1.5}

\def\Xint#1{\mathchoice
{\XXint\displaystyle\textstyle{#1}}%
{\XXint\textstyle\scriptstyle{#1}}%
{\XXint\scriptstyle\scriptscriptstyle{#1}}%
{\XXint\scriptscriptstyle\scriptscriptstyle{#1}}%
\!\int}
\def\XXint#1#2#3{{\setbox0=\hbox{$#1{#2#3}{\int}$ }
\vcenter{\hbox{$#2#3$ }}\kern-.6\wd0}}
\def\ddashint{\Xint=}
\def\dashint{\Xint-}

\def\E{e}\def\I{i}

We propose a powerful method that allows us to numerically calculate the `free' \cite{vdn91} additive, multiplicative and compressive convolution of a large class of probability measures. We see this method as complementing the symbolic techniques previously developed in \cite{re08} for so-called algebraic measures, \textit{i.e.}, measures whose Cauchy transforms are algebraic.

Using the method developed in this paper, we can, for example, compute \textit{with spectral accuracy} the free additive convolution of the semi-circle and the Gaussian which arises in \cite{bryc2006spectral}  (see Figure \ref{SemiPlusGauss}); or the free compression of the Gaussian which arises  in \cite{belinschi2010normal} (see Figure \ref{GaussCompress}); or even  the free additive convolution of the Gaussian with the counting measure on a single realization of a Gaussian Orthogonal Ensemble as a way to get insight on the rate of convergence to the asymptotic result in \cite{bryc2006spectral} (see Figures \ref{SemiPlusGauss} \& \ref{DSemiPlusGauss}). We go well beyond these simple examples and hope that the proposed method allows practitioners to experiment with free probability convolutions so that  they may find new applications of the underlying theory.

We consider the  free convolution operations on measures $\mu_A$ and $\mu_B$ (and compression factor $\alpha \in (0,1)$) listed in the first column of Table \ref{tab:operations}.
Each operation takes in one or two measures, and returns a new measure.  What is known in each case is a relationship in transform space; i.e., there are transforms \cite{voiculescu1986addition,voiculescu1987multiplication,vdn91,nica1996multiplication} $R_\mu(y)$ and $S_\mu(y)$ so that the convolution operation can be represented simply as in the second column of Table \ref{tab:operations}.

The {\it Cauchy transform} of a measure $\mu$ on the real line is defined as:  $$G_\mu(z)=\int\f{\ud \mu(x)}{z-x}  \qquad \textrm{for } z \notin \supp \mu.$$
The key observation is that each transform $R_\mu$ and $S_\mu$ can be expressed in terms of the functional inverse of the Cauchy transform $G_\mu^{-1}$  (which we refer to as the {\em inverse Cauchy transform}).  Therefore, we reduce the problem to the following two subtasks:

\begin{enumerate}
	\item calculate the inverse Cauchy transform of the input measures pointwise; and
	\item recover the output measure from knowledge of its inverse Cauchy transform.
\end{enumerate}

\subsection{Types of measures, their utility and the key underlying idea}

\begin{table}
\centering
\begin{tabular}{p{3.55cm}ll}
\otoprule
\makebox[3.25cm][c]{{\bf Operation}} &  \makebox[4.55cm][c]{{\bf Transform Operation}} &  \makebox[4.25cm][c]{{\bf Key Transform}}\\
\otoprule
\multirow{2}*[-1ex]{\smallbox{\minitab[c]{Free Addition \\ $\mu_{C} = \mu_{A} \boxplus \mu_{B}$}}} &
\multirow{2}*[-2.25ex]{\smallbox{\minitab[l]{$R_{\mu_C}(y) = R_{\mu_A}(y) + R_{\mu_B}(y)$}}}
& \multirow{2}*[0ex]{\smallbox{\minitab[l]{$G_{\mu}(z) = \displaystyle \int\dfrac{\ud \mu(x)}{z-x},$ \\[0.35cm]
$R_{\mu}(y) = G_{\mu}^{-1}(y) - \dfrac{1}{y}$}}}\\[1.50cm]

\midrule
\multirow{2}*[-1ex]{\smallbox{\minitab[c]{Free Multiplication \\ $\mu_{C} = \mu_{A} \boxtimes \mu_{B}$}}} &
\multirow{1}*[-2.25ex]{\smallbox{\minitab[l]{$S_{\mu_C}(y) = S_{\mu_A}(y) \, S_{\mu_B}(y)$}}}
& \multirow{2}*[0ex]{\smallbox{\minitab[l]{$T_{\mu}(z) = \displaystyle \int\dfrac{x \ud \mu(x)}{z-x},$ \\[0.35cm]
$S_{\mu}(y) = \dfrac{1+y}{y} \cdot \dfrac{1}{T_{\mu}^{-1}(y)}$}}}\\[1.50cm]
\midrule
\multirow{2}*[-1ex]{\minitab[c]{Free Compression \\ $\mu_{C} = \alpha \bxd \mu_A$}} &
\multirow{1}*[-2.25ex]{\smallbox{\minitab[l]{$R_{\mu_C}(y) = R_{\mu_A}(\alpha y) $}}}
& \multirow{2}*[0ex]{\smallbox{\minitab[l]{$G_{\mu}(z) = \displaystyle  \int\dfrac{\ud \mu(x)}{z-x},$\\[0.35cm]
$R_{\mu}(y) = G_{\mu}^{-1}(y) - \dfrac{1}{y}$}}}\\[1.50cm]
\otoprule
\end{tabular}
\caption{Free convolution operations considered in this paper.}
\label{tab:operations}
\end{table}

We will focus on the following types of measures:

\begin{Def}
	A measure $\mu$ is a {\it smoothly decaying measure} if it has the form
	 $$\ud\mu(x) = \psi(x)  \ud x,$$
where $\psi \in C^1(-\infty,\infty)$, $\psi'$ has  bounded variation and $\psi(x) = {\alpha \over x} + O(x^{-2})$ as $x \rightarrow \pm \infty$ for some constant $\alpha$.   A {\it Schwartz measure} is a smoothly decaying measure such that $\psi$ is Schwartz: $\psi \in C^\infty(-\infty,\infty)$ and 
	$$\int_{-\infty}^\infty \psi(x) |x|^k dx < \infty$$
for $k = 0,1,2,\ldots$. 
\end{Def}

\begin{Def}\label{JacobiMeasureDef}
	A measure $\mu$ is a {\it Jacobi measure} if it has the form
	$$d \mu = \psi(x) (x-a)^\alpha (b - x)^\beta dx$$
where $\alpha,\beta > -1$,  $\psi \in C^1[a,b]$ and $\psi'$ has bounded variation.   $\mu$ is {\it precisely a Jacobi measure} if it is a Jacobi measure such that $\psi(a),\psi(b) \neq 0$.
	A {\it square root decaying measure}  is a Jacobi measure with $\alpha = \beta = {1 \over 2}$.  
\end{Def}

\begin{Def}
	A measure $\mu$ is a {\it half square root/smoothly decaying measure} if it is supported on $(a,\infty)$ (similarly, $(-\infty,b)$ and has the form
	$$d \mu = \psi(x) \sqrt{x-a}dx$$
where $\psi \in C^1[a,\infty)$, $\psi'$ has  bounded variation and $\psi(x) = {\beta \over x} + O(x^{-2})$ as $x \rightarrow + \infty$ for some constant $\beta$. A measure is {\it precisely half square root decaying/Schwartz} if $\psi \in C^\infty[a,\infty)$, $\psi(a) \neq 0$ and 
	$$\int_{a}^\infty \psi(x) \sqrt{x - a} |x|^k dx < \infty$$
for $k = 0,1,\ldots$
\end{Def}

In this paper, the class of \textit{admissible measures} (see Section \ref{InvC}) are measures for which the inverse Cauchy transform (and hence the $R$ or $S$ transforms) can be accurately computed pointwise on an appropriate domain:
	
\begin{Def}
	A measure is {\it admissible} if it is of the following type:
\begin{enumerate}
	\item smoothly decaying measures,
	\item Jacobi measures,
	\item half square root/smooth decaying measures,
	\item point measures or
	\item finite combinations of the above with compact support.
\end{enumerate}
\end{Def}
The class of \textit{invertible measures} are a subset of the class of admissible measures:
\begin{Def}
	A measure is {\it invertible} if its Cauchy transform is single-valued off its support and it is one of the following:
\begin{itemize}
\item precisely square-root decaying measure,
\item Schwartz measure, or
\item half precisely square-root/Schwartz measure.
\end{itemize}
\end{Def}

Invertible measures are a class of measures for which we can guarantee \textit{recovery} of the output measure accurately from knowledge of its inverse Cauchy transform.  The theory of Section \ref{regularity} states  broad conditions for which  the result of a free probability operation is an invertible measure.    The utility of the invertible measures can be discerned from Table \ref{tab:rep of measures}.

The key idea behind the proposed method is that invertible measures that are represented via a Chebyshev or Fourier series expansion as in the second column of Table \ref{tab:rep of measures}, have Cauchy transforms whose series expansions are closely related, as listed in the third column of Table \ref{tab:rep of measures}. Thus, given a series truncation, we can efficiently compute the inverse Cauchy transform $G_{\mu}^{-1}(y_{i})$ at $y_{i}$ by a companion matrix method. This knowledge of the inverse Cauchy transform $G_{\mu}^{-1}(y_i)$ at points $\{y_{i}\}_{i=1}^{m}$ coupled with the relationship (valid for $y$ in the image of $G_\mu$ for invertible measures):
$$G_{\mu}(G_{\mu}^{-1}(y)) = y,$$
implies that the desired series expansion coefficients $\{\psi_{k}\}_{i=1}^{n}$ for the Cauchy transform representation in the third column of Table \ref{tab:rep of measures} can be recovered by solving the Vandermonde system defined by:
$$ G_{\mu}(G_{\mu}^{-1}(y_{i})) = y_{i} \qquad \textrm{ for } i = 1, \ldots, m > n.$$
This yields  \algref{smoothlydecaying} and \algref{sqrtsupp} for smoothly-decaying and square-root decaying measures, respectively. Choosing $n$ and $m$ appropriately yields the desired level of accuracy. Once these expansion coefficients are computed, we recover the measure $\mu$ via the series expansion in the second column of Table \ref{tab:rep of measures}.

The recognition that the class of invertible measures has a nice series representation for \textit{both} the measure and its Cauchy transform is an important ingredient of the method; this insight originated in \cite{SOHilbertTransform} and might be of independent interest to free probabilists. Representing the measures via another basis that yields a sparser series representation of the measure but that does not yield a sparse, directly computable and invertible, series representation of Cauchy transform does not lead to an algorithm for computing the inverse Cauchy transform, thereby stalling progress {in the development of a numerical approach}.

The paper is organized as follows. In Section \ref{regularity}, we discuss the analytic properties of the Cauchy transform that motivate the construction of the numerical method, and guarantee convergence.  In Section \ref{InvC}, we describe a numerical approach for the first sub-task, \textit{i.e.}, the calculation the inverse Cauchy transform for several types of admissible measures that arise in practice.  We then solve the inverse problem in Section \ref{recoverM}: we develop an algorithm to recover an unknown measure based on pointwise evaluation of its inverse Cauchy transform.  In each stage, we achieve spectral accuracy.  In the remaining sections, we apply this numerical algorithm to  free additive, multiplicative and compressive convolution.

\begin{table}
\centering
\begin{tabular}{p{3.25cm}p{4.35cm}p{6.25cm}}
\otoprule
\makebox[3.25cm][c]{Type} & \makebox[4.3cm][c]{Measure} & \makebox[5.55cm][c]{Cauchy transform}\\
\otoprule
\multirow{3}*[3ex]{\smallbox{\minitab[c]{Square-Root \\Decaying \\(e.g. Semi-Circle)}}} &
\tinybox{\minitab[l]{$d\mu(x) = \psi(x) \dfrac{2 \sqrt{x-a}\sqrt{b-x}}{b-a} dx,$ \\[0.35cm]
            $\psi(x)  = \sum_{k=0}^{\infty} \psi_{k}\, U_{k}(M_{(a,b)}^{-1}(x))$\\[0.35cm]}}
&\tinybox{\minitab[l]{$G_{\mu}(z) = \pi \sum_{k=1}^{\infty} \psi_{k-1} J_{+}^{-1}(M_{(a,b)}^{-1}(z))^{k},$ \\[0.35cm]
where $J_+^{-1}(z) =  z - \sqrt{z-1} \sqrt{1 + z}$ and, \\[0.35cm]
$M_{(a,b)}(z) = {a + b \over 2} + {b -a \over 2} z$}}
\\[0.35cm]
\midrule
\multirow{3}*[3ex]{\smallbox{\minitab[c]{Half Square Root \\/Smoothly Decaying}}} &
\tinybox{\minitab[l]{$d\mu(x) = \psi(x) \dfrac{2 \sqrt{x-a}}{1 + x - a} dx,$ \\[0.35cm]
            $\psi(x)  = \sum_{k=0}^{\infty} \psi_{k}\, U_{k}(M_{(a,\infty)}^{-1}(x))$\\[0.35cm]}}
&\tinybox{\minitab[l]{$G_{\mu}(z) = \pi \sum_{k=1}^{\infty} \psi_{k-1} \left[J_{+}^{-1}(M_{(a,\infty)}^{-1}(z))^{k}-1\right],$ \\[0.35cm]
where $J_+^{-1}(z) =  z - \sqrt{z-1} \sqrt{1 + z}$ and, \\[0.35cm]
$M_{(a,\infty)}(x) = a + { 1 + x \over 1 - x}$}}
\\[0.35cm]
\midrule
\multirow{3}*[5ex]{\smallbox{\minitab[c]{Smoothly \\Decaying  \\(e.g. Gaussian)}}} &
\tinybox{\minitab[l]{$d\mu(x) = \psi(x) dx$, \\[0.35cm]
$\psi(x) = \sum_{k=-\infty}^{\infty} \psi_{k}  u(x)^{k}$,\\[0.35cm]
where $\psi_{k} = \overline{\psi}_{-k}$ and,  \\[0.35cm]
 $u(x) = \dfrac{i-x}{i+x}$}}
&
\tinybox{\minitab[l]{$G_{\mu}(z)=-\sum_{k=0}^{\infty}(-1)^{k}\psi_{k} \,\,\,+ $ \\[0.35cm]
$- 2\pi \begin{cases} \sum_{k=0}^{\infty} \psi_{k}\, u(z)^{k}, \Im(z) > 0 \\[0.35cm]
\sum_{k=-1}^{-\infty} \psi_{k}\, u(z)^{k}, \Im(z) < 0 \\[0.35cm]
\end{cases},$\\[0.35cm]
where $u(z) = \dfrac{i-z}{i+z}$}}
\\[0.35cm]
\otoprule
\end{tabular}
\caption{Series representation of invertible measures and their associated Cauchy transforms.}
\label{tab:rep of measures}
\end{table}

%
%

\section{Regularity properties of free convolution and its implication}\label{regularity}

We think of admissible measures as candidate `input' measures that we would like to convolve using the operations in Table \ref{tab:operations}.   In this viewpoint, invertible measures are the generic `output' measures that result from the convolution of admissible  `input' measures.  Table \ref{tab:rep of measures} lists the class of invertible measures; recall that these are measures that can be recovered accurately from knowledge of their inverse Cauchy transform. In contrast, admissible measures are those for which we can compute the inverse Cauchy transform accurately.

The semi-circle and Gaussian measures are both invertible and admissible; the uniform measure on an interval and the (discrete) point measure are admissible but not invertible.  Invertible measures are thus a proper subset of the class of admissible measures. Might this be a shortcoming of our proposed method? We assert otherwise and argue why the mathematics of free convolution gives us license to carve out the smaller class of invertible measures from the larger class of admissible measures.

Simply put, the free convolution of two admissible measures, under broad conditions,  results in an invertible measure. An important implication is that we can predict the form of the convolved measure and apply a suitable algorithm (see Section \ref{recoverM}) for recovering the measure from its inverse Cauchy transform.

In the discussion of the theory, we focus on free addition.  The {\it $R$-transform}, defined as $$R_\mu(y):=G_\mu^{-1}(y)-1/y,$$ is the analogue of the logarithm of the Fourier transform for free additive convolution. The free additive convolution of \pro measures on the real line   is denoted by the symbol $\bxp$ and can be characterized as follows.

Let $A_n$ and $B_n$ be independent $n \times n$ symmetric (or Hermitian) random matrices that   are invariant, in law, by conjugation by any orthogonal (or unitary) matrix. Suppose that, as $n \longrightarrow \infty$, $\mu_{A_{n}} \rightharpoonup \mu_{A}$ and $\mu_{B_{n}} \rightharpoonup \mu_{B}$. Then, free probability theory \cite{voiculescu1986addition,vdn91} states that $\mu_{A_n + B_n} \rightharpoonup \mu_{A} \bxp \mu_{B}$, a \pro measure which
 can be characterized in terms of the $R$-transform as
\begin{equation}\label{Radd}
	R_{\mu_A\bxp\mu_B}(y)= R_{\mu_A}(y)+R_{\mu_B}(y).
\end{equation}
Rearranging \eqref{Radd}, we find that
	$$G_{\mu_A \bxp \mu_B}^{-1}(y)  = G_{\mu_A}^{-1}(y) +  G_{\mu_B}^{-1}(y) -{1 \over y}.$$

In the following two theorems, we use this complex analytical statement to derive conditions for which invertible square root decaying measures and Schwartz measures are guaranteed to arise from free addition.    Refer to Appendix~\ref{cauchyprops} for a list of related properties of Cauchy transforms of probability measures.

\begin{Def}
	We denote the image of the Cauchy transform of a measure over the {\it extended} complex plane by
	$G_\mu(\C)$.
This includes both limiting values for $x \in \supp \mu$.  In other words:
	$$G_\mu(\C) =  \{y : \hbox{for all $\epsilon$ there exists $z \in \C$ satisfying $|{G_\mu(z) - y}| < \epsilon$}\}.$$
\end{Def}

	\begin{figure}[tb]
  \begin{center}
	\includegraphics[width=.7\linewidth]{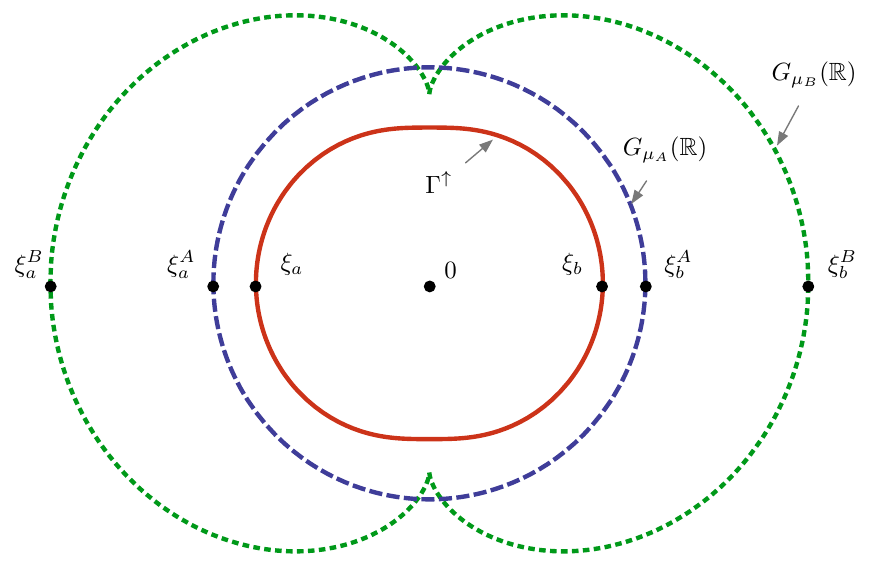}
  \end{center}
  \caption{\label{RealDomain}A plot depicting the image of the real line of $G_{\mu_S}$ (dashed), $G_{\mu_4}$ (dotted) and $G_{\mu_S \bxp \mu_4}$ (plain), where $\mu_S$ is the semicircle distribution and $\mu_4$ is the equilibrium measure with potential $V(x) = x^4$ (see Section~\ref{SemiPlusQuartic}).  The image of the output measure lies clearly inside the image of the two input measures.  
}  

\end{figure}

\begin{Th}\label{sqrtgamma}
	Suppose $\mu_A$ is a precisely square root decaying invertible measure  and $\mu_B$ is a Jacobi measure whose Cauchy transform is single-valued.  Then $\mu_A \bxp \mu_B$ is precisely square root decaying and invertible, and 
	$$G_{\mu_A \bxp \mu_B}(\C) \subset 	G_{\mu_A}(\C) \cap G_{\mu_B}(\C).$$
This subset is strict.  Moreover, for $y \in G_{\mu_A}(\C) \cap G_{\mu_B}(\C)$,
	$$\sgn \Im G_{\mu_A \bxp \mu_B}^{-1}(y) \neq \sgn \Im y$$
if and only if $y \in G_{\mu_A \bxp \mu_B}(\C)$.

\end{Th}
\begin{proof}

\begin{remark}
The statement 	$$G_{\mu_A\bxp\mu_B}(\C) \subset G_{\mu_A}(\C)\cap G_{\mu_B}(\C)$$ first appeared in Proposition 4.3 of \cite{voiculescu1994analogues}. It is a direct consequence of the subordination  of the functions expressed in \cite{b98}.  However, we re-derive it below in a different manner.  
\end{remark}

	We restrict our attention to  the upper half plane $\C^+$,  where $\C^\pm = \{y : \sgn \Im y  = \pm 1\}$, as commuting with complex conjugate proves the other case (see Proposition~\ref{cauchyconj}).  Denote $\mu_C = \mu_A \bxp \mu_B$.  Define $\xi^A_a = G_{\mu_A}(\min \supp \mu_A)$,   $\xi^A_b = G_{\mu_A}(\max \supp \mu_A)$, and $\xi^B_a$ and $\xi^B_b$ similarly.  
Define
	$$g(y) = G_{\mu_A}^{-1}(y) +G_{\mu_B}^{-1}(y) - {1 \over y}.$$
Because $G_{\mu_A}$ and $G_{\mu_B}$ are single-valued, we know that $g$ is analytic in the interior of $G_{\mu_A}(\C) \cap G_{\mu_B}(\C)$.  Moreover, from free probability theory, we know that $g(y) = G_{\mu_A \bxp \mu_B}^{-1}(y)$ in a neighborhood of zero.

 The proof consists of the following: (1) showing that there exists a single curve $\Gamma^\uparrow$  on which $g$ is real-valued in the interior of $G_{\mu_A}(\C^-) \cap G_{\mu_B}(\C^-) \subset \C^+$, (2) showing that $\Gamma^\uparrow$ intersects the real axis at $0 < \xi_b < \min(\xi^A_b,\xi^B_b)$ and $\max(\xi^A_a,\xi^B_a) < \xi_a <0$, (3) showing that $g$ has a first order turning point at $\xi_a$ and $\xi_b$.  By uniqueness, it will follow that $\xi_a = \xi_a^C = G_{\mu_C}(\min \supp \mu_C)$ (and similarly $\xi_b = \xi_b^C$).     See Figure~\ref{RealDomain} for a depiction.

	(Step 1) 
 Consider $\nu$ on the boundary of  $G_{\mu_A}(\C^-) \cap G_{\mu_B}(\C^-)$, which must be continuous and bounded (by Proposition~\ref{continuousboundary}).  Being on the boundary implies either $G_{\mu_A}^{-1}(\nu)$ or $G_{\mu_B}^{-1}(\nu)$ is real; assume the former without loss of generality.  By Proposition~\ref{dominance}, we have
	$$\Im g(\nu) = \Im\left(G_{\mu_B}^{-1}(\nu) - {1 \over \nu}\right) > 0;$$
in other words, $g$ on this boundary has strictly positive imaginary part.  On the other hand, in a neighbourhood of zero $\Im g(y) \sim  \Im {1 \over y} \rightarrow - \infty$, hence there is a point $\zeta \in G_{\mu_A}(\C^-) \cap G_{\mu_B}(\C^-)$ such that $g(\zeta)$ is real-valued.   By analyticity, it follows that there exists a {\it curve} $\Gamma^\uparrow$ on which $g$ is real-valued.

	Because $G_{\mu_A}(\C^-) \cap G_{\mu_B}(\C^-)$ is bounded and restricting $\Gamma^\uparrow$, $\Gamma^\uparrow$  must either connect with the real line or approach the singularity of $g$ at zero. But we know that
	\begin{equation}\label{poleplusanalytic} 
		g(y) = {1 \over y} + \hbox{analytic}.
		\end{equation}
It follows that $\Gamma^\uparrow$ must connect with the real axis away from zero:  there exists $c < 0$ such that
	$$\Im g(y) < c \Im y$$
in a neighborhood of zero,  and hence no real-valued curve can be present in this neighborhood.   Moreover, $\Gamma^\uparrow$ must connect with the real axis at two points $\xi_a$ and $\xi_b$: one to the left and one to the right of zero.  Otherwise, a closed loop on which $\Im g$ is identically zero would be formed, falsely implying that $g$ is a real constant.  By the same logic, it is the only curve inside $G_{\mu_A}(\C^-) \cap G_{\mu_B}(\C^-)$ such that $g$ is real-valued.  Therefore, $\Gamma$ divides $G_{\mu_A}(\C^-) \cap G_{\mu_B}(\C^-)$ into an interior region $D$  such that $\Im g(y) > 0$ for all $y$ in $D$, and an exterior region such that $\Im g(y) < 0$.

(Step 2) We now show that  $\xi_b < \min(\xi^A_b,\xi^B_b)$  (similar logic shows $\max(\xi^A_a,\xi^B_a) < \xi_a$). 

\begin{enumerate}
	\item Case $\xi^A_b \leq \xi^B_b$:  We have (using Proposition~\ref{sqrtturning} and Proposition~\ref{derjensen})
	$$g'(\xi_b^A) = {G_{\mu_A}^{-1}}'(\xi^A_b)  + {G_{\mu_B}^{-1}}'(\xi^A_b) + {1 \over {\xi^A_b}^2} = {G_{\mu_B}^{-1}}'(\xi^A_b) + {1 \over {\xi^A_b}^2} >0.$$
However,
$g'(y) < 0$ near zero by \eqref{poleplusanalytic}.  Thus there exists a point in $(0,\xi^A_b)$ where $g'$ vanishes.  Combining the uniqueness of the curve on which $g$ is real and the fact that any turning point of a real-valued analytic function has another real-valued curve emanating into the complex plane, this turning point must be $\xi_b$.  

	\item Case $\xi^B_b < \xi^A_b$:  We know that $\beta > 0$, as otherwise $\xi^B_b = \infty$ (Proposition~\ref{sqrtturning}).  We have that $G_{\mu_B}(y) =  C + (y-\xi^B_b)^\beta + o(y - \xi^B_b)^\beta$ where $C$ is real-valued (Proposition~\ref{sqrtturning}).  Suppose that $\xi_b = \xi^B_b$, and hence $\Gamma^\uparrow$ approaches $\xi^B_b$.      For $\nu$ on the boundary of $G_{\mu_B}^{-1}(\C^-)$, we have
	 $$\Im g(\nu) = \Im\left(G_{\mu_A}^{-1}(\nu) - {1 \over \nu}\right) \sim c \Im {\nu}$$
 for $c \neq 0$, because $G_{\mu_A}^{-1}(y) - {1 \over y}$ is analytic at $\xi^B_b$ and does not have a turning point there (Proposition~\ref{derjensen}).  Because $\Im g(\zeta)$ vanishes on $\Gamma^\uparrow$, it must be bounded away from the boundary of $G_{\mu_B}^{-1}(\C^-)$ by a fixed angle.  But near $\xi^B_b$, the $(y-\xi^B_b)^p$ term dominates, imposing non-negative imaginary part away from the boundary of $G_{\mu_B}^{-1}(\C^-)$.  Thus we have a contradiction, and $\xi_b < \xi^A_b$.  

\end{enumerate}



Since $g$ is equal to $G_{\mu_C}^{-1}$ in a neighborhood of the origin and single-valued in $D$, we see that $G_{\mu_C}^{-1} = g$ inside all of $D$.  By uniqueness of real-valued curves, it follows that $\Gamma^\uparrow$ is the boundary of $G_{\mu_C}(\C^-)$.  

Note that
	\begin{equation}\label{noturning}
		g'(y) = {1 \over G_{\mu_C}'(G_{\mu_C}^{-1}(y))}.
	\end{equation}
Thus $g'$ cannot vanish when $y$ has nonzero imaginary part.  Moreover, it cannot vanish for $\xi_a < y < \xi_b$ as that would imply an additional real-valued curve.  Thus $g$ is  a conformal map from $D$ to $\C^-$.   The inverse function theorem  implies that $g^{-1}$ exists and is single-valued.  Moreover it is analytic everywhere, even (by analytic continuation) on the branch cut,   except at the branch points $a = a_C =  g^{-1}(\xi_a)$ and $b = b_C = g^{-1}(\xi_b)$.  Plemelj's lemma implies that
	$$\psi_C(x) = -{1 \over 2 \pi i} (G_{\mu_C}^+(x) - G_{\mu_C}^-(x)) = {1 \over \pi } \Im G_{\mu_C}^-(x) = {1 \over \pi } \Im g^{-1}(x),$$
hence $\psi_C$ itself is analytic (hence H\"older-continuous) for $x \in (a,b)$.  

Finally consider the endpoints $\xi_a$ and $\xi_b$, where $g$ necessarily has a turning point.  Since only two real-valued curves emanate from the turning point, it is necessarily first order (i.e., behaves like $c_0 + c_2 (y - \xi_a)^2$ for $c_2 \neq 0$).  Thus the inverse function theorem implies that $G_{\mu_C}$ has the convergent series
	$$G_{\mu_C}(z) = c_0 + c_{1 \over 2} \sqrt{z - a} + c_1 (z-a) + c_{3 \over 2} (z-a)^{3\over 2} + \cdots$$ 
where the real-valuedness of $G_{\mu_C}(z)$ for $z > a$ imposes that the $c_k$ are real.  The polynomial terms vanish from $\Im G_{\mu_C}(x)$, leaving an analytic function times $\sqrt{z-a}$.   Thus $\mu_C$ is a precisely square root decaying, invertible measure.

\end{proof}

A similar result showing that Schwartz measures dominate other behavior now follows:

\begin{Th}\label{smoothgamma}
	Suppose $\mu_A$ is an invertible Schwartz measure and $\mu_B$ is single-valued and either a Schwartz measure or  a compactly supported admissible measure.  Then $\mu_A \bxp \mu_B$ is an invertible Schwartz measure, and 
	$$G_{\mu_A \bxp \mu_B}(\C) \subset 	G_{\mu_A}(\C) \cap G_{\mu_B}(\C).$$
This subset is strict, except at zero.   Moreover, for $y \in G_{\mu_A}(\C) \cap G_{\mu_B}(\C)$,
	$$\sgn \Im G_{\mu_A \bxp \mu_B}^{-1}(y) \neq \sgn \Im y$$
if and only if $y \in G_{\mu_A \bxp \mu_B}(\C)$.

\end{Th}

\begin{proof}
	Again, let $\Im y > 0$ (commuting with complex conjugate proves the other case).  

By the same logic as before, there exists a curve $\Gamma^\uparrow$ inside $G_{\mu_A}(\C) \cap G_{\mu_B}(\C)$, but now it must pass through zero, as $G_{\mu_A}$ is strictly in the upper half plane.  We must verify that there is not a second curve on which $g$ is real-valued (also passing through zero).  By Proposition~\ref{smoothdecaydecay}, we have
	$$g(y) \sim {1 \over y}  + g_0 + g_1 y + \cdots,$$
where $g_k$ are all real.  
This means that both $\Gamma$ and any other curve $\Upsilon$ of real-valuedness must be asymptotic to the real line (on which ${1 \over y}$ is real-valued).  We can appeal to the Poisson kernel with $\Im g$ evaluated on both $\Gamma$ and $\Upsilon$, with small contours connecting these to avoid the singularity at zero.  Letting these contours tend to zero causes $\Im g$ to tend to zero (since $\Im {1 \over y}\rightarrow 0$ between $\Gamma$ and $\Upsilon$).  This  implies that $\Im g$ is identically zero between $\Gamma$ and $\Upsilon$, giving a contradiction.

By uniqueness of the real-valued curve, we have that $g$ has no turning points on $\Gamma$.  Furthermore, by \eqref{noturning}, we know that $g$ has no turning points inside $\Gamma$.  Thus $g$ is a conformal map; hence, $g^{-1}(z) = G_{\mu_C}^{-1}(z)$ exists and is analytic in $\C^-$.  Plemelj's lemma implies that
	\begin{equation*}\label{eq:recoverC}
		\psi_C(x) = {1 \over \pi } \Im G_{\mu_C}^-(x) = {1 \over \pi } \Im g^{-1}(x)
		\end{equation*}
is analytic and non-zero.

We only have to show that it has the correct decay at infinity.  
This follows since
	$$g^{-1}(z)  \sim {1 \over z} + {\alpha_{-2} \over z^2} + {\alpha_{-3} \over z^3} + \cdots,$$
for $\alpha_k$ real (by adding the asymptotic expansion of $G_{\mu_A}^{-1}$ and either the asymptotic expansion for Schwartz measures or the  Laurent series with a simple pole for compactly supported measures of $G_{\mu_B}^{-1}$); thence, 
	$$\Im g^{-1}(x) \sim {0 \over x} + {0 \over x^2} + {0 \over x^3} + \cdots.$$


%

\end{proof}

\begin{remark}
	The conditions in the preceding two theorems are far from exhaustive, and the numerical scheme below works in practice for many other free convolution problems.   Some examples include the free addition of two step measures, which results in a precisely square root decaying measure.  On the other hand, we have an example of  a Jacobi measure with $\alpha = \beta = 5/2$ convolved with itself that does not appear to  result in a precisely square root decaying measure; rather, it appears to have {\it linear decay}.  Moreover, if the input measures are smoothly decaying but with only algebraic decay, the convolved measure has, apparently, only algebraic decay.   We will not attempt to  generalize   the preceding proofs to other classes of measures here, or to free multiplication.  
\end{remark}

\section{Computation of inverse Cauchy transforms}\label{InvC}

For brevity, we omit the details for half square root/smoothly decaying measures below, as they can be treated very similarly to square root decaying measures.  We include the relevant formul\ae\ in Table \ref{tab:rep of measures}.

We also omit the formul\ae\ for the Cauchy transform of expansions in Jacobi polynomials --- which are expressible in terms of hypergeometric functions \cite{elliott71} --- except for the simplest case of square root decaying measures.  Inversion of the Cauchy transform of Jacobi measures can be accomplished using the approach advocated in Section~\ref{combomeasures}.


We note that the formul\ae\ for the Cauchy transforms below follow from Plemelj's lemma \cite{SIE}: i.e., if $\ud \mu = \psi \ud x$ for suitably smooth $\psi$, then
	$$\phi^+(x) - \phi^-(x) = - 2 \pi i \psi(x)\hbox{ and }\phi(\infty) = 0$$
if and only if
$\phi = G_\mu$, where $\phi^+$ denotes the limit in the complex plane from above and $\phi^-$ denotes the limit from below.

\begin{remark}
	In most presentations of Plemelj's lemma, ``suitably smooth'' means H\"older continuous, eg.~\cite{SIE}.  In fact, the lemma continues to hold true for all measures of the form $\psi(x) dx$ where  $\psi \in {\rm L}^p[{\Bbb R}]$ for $1 < p < \infty$, see, for example, Section 7.1 in \cite{d00}.  
\end{remark}

%

\subsubsection{Computing the Cauchy transform and its function inverse of smoothly decaying measures}
\label{sec:cauchy compute smooth}

Consider a  smoothly decaying measure  of the form
	 $\ud\mu(x) = \psi(x)  \ud x.$
 Because $\psi\left({\I {1 - z \over 1 + z}}\right)$ has an absolutely convergent Laurent series, we  can expand
	\begin{equation}\label{realseries}
		\psi(x)  = \sum_{k = -\infty}^\infty \psi_k \left({i - x \over i + x}\right)^k,
	\end{equation}
where $\psi_k = \bar \psi_{-k}$ (since $\psi$ is real-valued) and $\psi(\infty) = \sum_{k=-\infty}^\infty (-)^k \psi_k = 0$.  The Cauchy transform satisfies \cite{WeidemanHilbertTransform,SOHilbertTransform}
	\begin{equation}\label{smoothlycauchy}
		G_\mu(z) =-2 \pi i \left[\begin{cases}
	\sum_{k = 0}^\infty \psi_k \left({i - z \over i + z}\right)^k & \Im z > 0 \\
	-\sum_{k = -\infty}^{-1} \psi_k \left({i - z \over i + z}\right)^k & \Im z < 0	
	\end{cases} -\sum_{k = 0}^\infty (-)^k \psi_k\right].
	\end{equation}

If $\psi$ is $C^\infty(-\infty,\infty)$ and $\psi$ has a full asymptotic expansion that matches at $\pm \infty$ (e.g., Schwartz measures), then the series \eqref{realseries} converges spectrally quickly.  Moreover, we can rapidly compute the coefficients of the expansion by applying the FFT to the pointwise function samples  $\psi\left(i {1 - {\mathbf u}_m \over 1 +  {\mathbf u}_m}\right)$, where $ {\mathbf u}_m$ are $m$ evenly spaced points on the unit circle:
	$${\mathbf u}_m = \left[-1,e^{i \pi \left({2 \over m}-1\right)}, \ldots ,e^{i \pi \left(1-{2 \over m}\right)} \right].$$
Thus we take $m = 2n + 1$ and uniformly approximate 
	$$G_\mu(z) \approx -2 \pi i \left[\begin{cases}
	\sum_{k = 0}^n \psi_k \left({i - z \over i + z}\right)^k & \Im z > 0 \\
	-\sum_{k = -n}^{-1} \psi_k \left({i - z \over i + z}\right)^k & \Im z < 0	
	\end{cases} -\sum_{k = 0}^n (-)^k \psi_k\right].$$
For large $n$, $\psi_k$ are accurate to machine precision.

	Now consider the problem of computing $G_\mu^{-1}$.
	Note that
	$$G_\mu\left(i {1 - z \over 1 + z}\right) \approx -2 \pi i \left[\begin{cases}
	\sum_{k = 0}^n \psi_k z^k & | z | < 0 \\
	-\sum_{k = -n}^{-1} \psi_k z^k & |z | > 0	
	\end{cases} -\sum_{k = 0}^n (-)^k \psi_k\right].$$
We can therefore invert the approximation of $G_\mu$ using a companion matrix method.  In detail, we compute the eigenvalues $\{\lambda_1^+(y),\ldots,\lambda_{n}^+(y)\}$ of the matrix
	$$\begin{pmatrix}& &  &  {\psi_0 - \sum_{k = 0}^n (-)^k \psi_k  -{y\over - 2\pi i}\over \psi_n}   \\ 1 &  & &  {\psi_1 \over \psi_n} \\ & \ddots  &  & \vdots \\
															& & 1 & {\psi_{n-1} \over \psi_n}\end{pmatrix}.$$
  Similarly, we compute the eigenvalues $\{\lambda_1^-(y),\ldots,\lambda_{n}^-(y)\}$ of the matrix
	$$\begin{pmatrix}& &  &  {\psi_{-n}  \over { y \over 2 \pi i} +  \sum_{k = 0}^n (-)^k \psi_k}  \\ 1 &  & &  {\psi_{-1} \over { y \over 2 \pi i} +  \sum_{k = 0}^n (-)^k \psi_k} \\ & \ddots  &  & \vdots \\
															& & 1 &  {\psi_{-1}  \over { y \over 2 \pi i} +  \sum_{k = 0}^n (-)^k \psi_k}\end{pmatrix}.$$
Then
	$$G_\mu^{-1}(y) \approx i {1 - \lambda(y) \over 1 + \lambda(y)},$$
where
	 $$\lambda(y) = \{\lambda_i^+(y) : |\lambda_i^+(y)| \leq 1\} \cup \{\lambda_i^-(y) : |\lambda_i^-(y)| \geq 1\}.$$
The number of computed eigenvalues $\lambda(y)$ will match the true number of $G_{\mu}^{-1}(y)$ (for large enough $n$),  due to the uniform convergence of Taylor series and Roch\'e's theorem.  In particular, for invertible Schwartz measures there will be precisely one for all $y$ in $G_{\mu}({\mathbb C})$.


\subsubsection{Computing the Cauchy transform and its function inverse of square root decaying measures}

	Suppose that $\mu$ is a square root decaying measure:
	 $$\ud\mu(x) = \psi(x) {2 \sqrt{x-a} \sqrt{b - x} \over b -a}    \ud x.$$
(The definition of $\psi$ here is a constant multiple of the $\psi$ in Definition \ref{JacobiMeasureDef}.) We can represent
	$$\psi(M_{(a,b)}(x))  = \sum_{k = 0}^\infty \psi_k U_k(x),$$
where $U_k$ denote the Chebyshev polynomials of the second kind and $M_{(a,b)}$ is an affine transformation from the unit interval to $(-1,1)$:
$$M_{(a,b)}(x) = {a + b \over 2} + {b -a \over 2} x.$$
  Then the Cauchy transform satisfies
	\begin{equation}\label{sqrtcauchy}
		G_\mu(z) =\pi \sum_{k = 1}^\infty \psi_{k-1} J_+^{-1}(M_{(a,b)}^{-1}(z))^k,
	\end{equation}
where
	$$J_+^{-1}(z) =  z - \sqrt{z-1} \sqrt{1 + z}$$
is an inverse to the Joukowsky transform
	$$ J(w) = {1 \over 2} \left(w + {1 \over w}\right).$$
{Here, and throughout the paper, $\sqrt z$ has the standard principle branch.  Therefore, $J_+^{-1}$ has a branch cut along $[-1,1]$,  maps the slit plane ${\mathbb C} \backslash [-1,1]$ to the interior of the unit disk and satisfies $J_+^{-1}(\infty) = 0$.}

\begin{remark}
	We have not found this exact expression for the Cauchy transform of a square root decaying measure in the literature, though directly related expressions are in \cite{SOHilbertTransform,SOEquilibriumMeasure}.  In short, it follows from Plemelj's lemma and the fact that	
		$$\lim_{\epsilon \rightarrow^+ 0}  i{J_+^{-1}(x+i \epsilon)^k  - J_+^{-1}(x - i \epsilon)^k \over 2 } = U_{k-1}(x) \sqrt{1-x^2},$$
verifiable by the substitution $x = \cos \theta$ \cite{dlmf}.
\end{remark}

We can compute the coefficients $\psi_k$ whenever we can evaluate $\psi$ pointwise, and thence the Cauchy transform itself.  This is accomplished by first computing the expansion in terms of Chebyshev polynomials of the first kind
	\begin{equation}\label{eq:sqrtTexp}
		\psi(M_{(a,b)}(x)) \approx \sum_{k=0}^{n-1} \psi_k T_k(x),
	\end{equation}
which can be accomplished by applying the DCT to $\psi(M_{(a,b)}({\mathbf x}_n))$, where ${\mathbf x}_n$ are $n$ Chebyshev points of the second kind:
	$${\mathbf x}_n = J({\mathbf u}_{2 (n-1)})_{1 : n}.$$
We then transform the expansion \eqref{eq:sqrtTexp} to an expansion in terms of Chebyshev polynomials of the second kind using the formul\ae
	$$T_0(x) = U_0(x), T_1(x) = {U_1(x) \over 2}\hbox{ and }T_k(x) = {U_{k}(x) - U_{k-2}(x) \over 2}\hbox{ for $k = 2,3,\ldots$}$$
 This approximation will converge spectrally when $\psi \in C^\infty[a,b]$.

\subsubsection{Computing the  inverse  Cauchy transform of a square root decaying measure}

	We want to solve
	$$G_\mu(z) = y.$$
Since $J_+^{-1}(J(w)) = w$ for $w$ inside the unit circle, we have
	$$G_\mu\left(M_{(a,b)}(J(w))\right) \approx \pi \sum_{k = 1}^n \psi_{k-1} J_+^{-1}(J(w))^k = \pi \sum_{k = 1}^n \psi_{k-1} w^k.$$
We can thus solve $G_\mu\left(M_{(a,b)}(J(w))\right) = y$ to find $w(y)$ inside the unit circle using a companion matrix method (as above).  Then
	$$G_\mu^{-1}(y) \approx M_{(a,b)}( J(w(y))).$$

\subsubsection{Computing the Cauchy transform and its function inverse of a point measure}

Suppose $\ud \mu(x) = \delta(x- a) \ud x$.  Then its Cauchy transform is trivial:
	$$G_\mu(z) = \int {\ud \mu \over z- x} = {1 \over z - a}.$$
Its inverse is
	$$G_\mu^{-1}(y) =  {1 \over y}+ a.$$

\subsubsection{Computing the function inverse of the Cauchy transform for other compactly supported measures}\label{combomeasures}

	For simplicity, consider the case where $\mu$ is a sum of point measures, for example, the counting measure
	$$\ud \mu = {1 \over n} \sum_{i = 1}^n \delta(x - \lambda_i) \ud x$$
 of one realization of a $n \times n$ random symmetric matrix with eigenvalues $\lambda_1,\ldots,\lambda_n$.  The Cauchy transform can be computed directly using the previous approach, however, its inverse is no longer straightforward to compute.   To calculate the inverse, we surround the support of $\mu$ by an ellipse $E_{(a,b),r}$ in the complex plane, on which the Cauchy transform of the measure is smooth.   We then exploit analyticity of the Cauchy transform outside of this ellipse.
	
	Define an ellipse  $E_{(a,b),r}$ surrounding the interval $(a,b)$ as the image of the unit circle under the map
	$M_{(a,b)}(J(r w)),$  with inverse ${1 \over r}J_+^{-1}(M_{(a,b)}^{-1}(z))$.    We can then expand a function $g$ defined on $E_{(a,b),r}$ by
	\begin{equation}\label{eq:ellipseexp}
		g(M_{(a,b)}(J(r w))) = \sum_{k = - \infty}^\infty g_k w^k,
	\end{equation}
where the coefficients are computable numerically using the FFT as before.

%
%
%

	On and outside this ellipse, $G_\mu$ is analytic and vanishes at infinity, therefore $G_\mu(M_{(a,b)}(J(r w))$ is analytic inside the unit circle for $r<1$ and vanishes at zero.  Hence we can efficiently represent it in terms of its Taylor series:
	 $$G_{\mu}(z) = \sum_{k = 1}^{\infty} g_k \left[{1 \over r}J_+^{-1}(M_{(a,b)}^{-1}(z))\right]^k.$$
Analyticity of this sum implies that the expression holds true for $z$ outside $E_{(a,b),r}$ as well.  Mapping this sum back to the unit circle allows us to compute $G_{\mu}^{-1}$ using companion matrix methods.


\begin{remark}

Note that $r$ is a free parameter.  As $r$ approaches one, the ellipse approaches the interval $(a,b)$, which includes the support of $\mu$.  Since  $G_\mu$ generically has singularities on the support of $\mu$,  the convergence rate of the expansion \eqref{eq:ellipseexp} degenerates.  For $r$ small, the ellipse is too large and the region of validity for computing $G_{\mu}^{-1}$ shrinks.  For the numerical examples below, we fix $r$ arbitrarily ($r = .8$).  A better approach would be to exploit the connection with the closely related problem of optimizing the radius of circle used in numerical differentiation; a problem solved in \cite{b11}.

\end{remark}

\section{Recovering a measure from its inverse Cauchy transform}\label{recoverM}

	Using the preceding formul\ae\ and the expressions for the transforms below, we can successfully compute  the inverse Cauchy transform  $G_\mu^{-1}$  of some unknown measure $\mu$ pointwise, which will arise as the output of a free probability operation.  If $\mu$ is either a smoothly  or a square root decaying measure, we assert that, under broad conditions, the following algorithm will construct an accurate approximation to $\mu$:

\Algorithm{computemeasure} Compute measure from inverse Cauchy transform
	\given  $G_\mu^{-1}$ (accurate in $G_\mu(\C)$), point cloud $\mathbf y_M = (y_1,\ldots,y_M)$ in the upper half plane and the assumed form of the measure $\mu$ (smoothly or  square root decaying invertible measure);
	\return a representation $\mu$;
	\step Use \algref{choosepoints} to prune  $\mathbf y_M$ so that all points lie inside  $G_\mu(\C)$;
	\step If the desired form for $\mu$ is a smoothly decaying measure, use \algref{smoothlydecaying};
	\step Otherwise, if the desired form for $\mu$ is square root decaying measure, use \algref{sqrtsing}.


The first step of the algorithm is to assure that all sample points lie within $G_\mu(\C)$.  Motivated by Theorem~\ref{sqrtgamma} or \ref{smoothgamma}, we use the following algorithm:

\Algorithm{choosepoints}  Prune points
	\given  $G_\mu^{-1}$  and point cloud $\mathbf y_M = (y_1,\ldots,y_M)$;
	\return $\mathbf y_m$ (hopefully $ \subset G_\mu(\C)$);	
	\step Select the elements of $\mathbf y_M$ that satisfy $\sgn \Im y \neq \sgn \Im g(y)$ and for which $g(y)$ is single-valued.


%
%
%
%
%


If we assume the measure is smoothly decaying (guaranteed if the input measures satisfy the hypotheses of Theorem~\ref{smoothgamma}), then we know precisely the form of its Cauchy transform, but we do not know the relevant coefficients of the expansion.  The following algorithm computes these coefficients by applying least squares to the equation
	$$G_\mu(G_\mu^{-1}(y)) = y,$$
which is valid for $y \in G_\mu(\C)$.

\Algorithm{smoothlydecaying}  Compute smoothly decaying measure
	\given $G_\mu^{-1}$, point cloud $\mathbf y_m$ inside $G_\mu(\C) \cap\C^+$ and positive integer $n$;
	\return a representation of $\mu$ that is smoothly decaying;
	\step Compute $\psi_k$ by solving the following system in a least squares sense:
$$ -2 \pi i 	\left[\sum_{k = 1}^n \psi_k \left({i - G_{\mu}^{-1}(y_j) \over i + G_{\mu}^{-1}(y_j)}\right)^k  -\sum_{k = 1}^n (-)^k \psi_k\right] \approx y_j.$$
%
	\step Define $\psi_0 = - 2 \Re \sum_{k = 1}^n (-1)^k  \psi_k$ and $\psi_{-k} = \bar{\psi}_k$.  Then
	$$\ud\mu \approx \sum_{k = -n}^n \psi_k \left({i - x \over i + x}\right)^k \ud x.$$

We now prove that, under broad conditions on $\mathbf y_m$, this algorithm will converge to the true coefficients $\psi_k$.  	


\begin{Def}
	We say that a smoothly decaying measure $\mu_n$ converges in mapped ${\rm L}^2$ to $\mu$ if, for 
	$$d\mu_n = \psi_n(x) dx\hbox{ and } d \mu = \psi(x) dx,$$
we have (for the ${\rm L}^2$ norm on the unit circle)
	$$\left\|\psi_n\left(i {1 - z \over 1 + z}\right) - \psi\left(i {1 - z \over 1 + z}\right) \right\| \rightarrow 0.$$
	We say that a square root decaying measure $\mu_n$ converges in mapped ${\rm L}^2$ to $\mu$ if $(a,b) = \supp \mu = \supp \mu_n$,  and, for
	$$d\mu_n = \psi_n(x) \sqrt{x - a} \sqrt{b - x} dx\hbox{ and } d \mu = \psi(x)  \sqrt{x - a} \sqrt{b - x}  dx,$$
we have
	$$\left\|\psi_n\left(M_{(a,b)}(J(z))\right) - \psi\left(M_{(a,b)}(J(z))\right) \right\| \rightarrow 0.$$

\end{Def}

\begin{Th}\label{methconv}
	Suppose that $\mu$ is an invertible smoothly decaying measure, and $\{{\mathbf y}_m\}$ are a sequence of sets of $m$ points lying inside $G_\mu(\C) \cap \C^+$ which cover $G_\mu(\C) \cap \C^+$ as $m \rightarrow \infty$ at a sufficiently fast rate (see proof and Appendix C for precise definition).  Then there exists $m$ sufficiently large depending on $n$ so that the output of \algref{smoothlydecaying} converges in mapped ${\rm L}^2$ to $\mu$ as $n \rightarrow \infty$.
\end{Th}
\begin{proof}

	Because of symmetry, including $\bar{\mathbf y}_m$ in the least squares system will not alter the approximation of $\mu$.  Therefore, denote $[{\mathbf y}_m,\bar{\mathbf y}_m] = [y_1,\ldots,y_{2 m}]$.  Then
	 $$y_j = G_{\mu}\left(i {1 - z_j \over 1 + z_j}\right)$$
 for some (unknown) $z_j$ inside the unit circle.  Under this transformation, the least squares system takes the form
$$ -2 \pi i 	\sum_{k = 1}^n \psi_k (z_j^k - (-1)^k) \approx y_j.$$
This is a Vandermonde system, with an unusual distribution of points.  However, as $m \rightarrow \infty$, the points $z_j$ must cover the unit circle,   and therefore convergence follows from Corollary \ref{vandermondeconvergencevanish}.

\end{proof}

	We can adapt this approach to square root decaying measures as well; since, assuming that we know the support of the measure, we again know a precise form for its Cauchy transform.

\Algorithm{sqrtsing}  Compute  square root decaying measure
	\given $G_\mu^{-1}$, point cloud $\mathbf y_m$ inside $G_\mu(\C) \cap \{z : \Im z > 0\}$ and positive integer $n$;
	\return a representation of $\mu$ that is square root decaying;
	\step Compute  $(a,b) \approx \supp \mu$ using \algref{sqrtsupp};
	\step Compute (real-valued) $\psi_k$ by solving the following system in a least squares sense:
\begin{align*}
	\pi  \sum_{k = 1}^n \psi_{k-1} \Re J_+^{-1}\left(M_{(a,b)}^{-1}(G_{\mu}^{-1}(y_j))\right)^k  & \approx \Re y_j\hbox{ and }\cr
	\pi  \sum_{k = 1}^n \psi_{k-1} \Im J_+^{-1}\left(M_{(a,b)}^{-1}(G_{\mu}^{-1}(y_j))\right)^k  & \approx \Im y_j	,
 \end{align*}
%
where
	$$M_{(a,b)}(x) = {a + b \over 2} + {b - a \over 2} x;$$
	\step Then
	$$\ud\mu \approx  {2 \sqrt{x-a} \sqrt{b - x} \over b -a}  \sum_{k = 0}^\infty \psi_k U_k(M_{(a,b)}(x)) \ud x.$$

If $\supp \mu$ is calculated accurately, the convergence of \algref{sqrtsing} follows by the same logic as Theorem \ref{methconv}:
\begin{cor}\label{sqrtconv}
		Suppose that $\mu$ is an invertible smoothly decaying measure, and $\{{\mathbf y}_m\}$ are a sequence of sets of $m$ points lying inside $G_\mu(\C) \cap \C^+$ which cover $G_\mu(\C) \cap \C^+$ as $m \rightarrow \infty$ at a sufficiently fast rate (see proof and Appendix C for precise definition).  Then there exists $m$ sufficiently large depending on $n$ so that the output of \algref{sqrtsing} converges in mapped ${\rm L}^2$ to $\mu$  as $n \rightarrow \infty$.
\end{cor}
    Thus we are left with one last task: computing $\supp \mu$.

\Algorithm{sqrtsupp} Compute the support of a square root decaying measure
	\given the first derivative of $G_\mu^{-1}$ and points $(a_0,b_0)$ satisfying $a_0 < \xi_a < \xi_b < b_0$;
	\return an interval $(a,b)$ approximating the support of $\mu$;
	\step Compute $\xi_a$ and $\xi_b$ by solving ${G_\mu^{-1}}'(y) =0$ using bisection, in the intervals $[a_0,0]$ and $[0,b_0]$.
	\step Set $(a,b) = (g(\xi_a),g(\xi_b))$.

In the next section, we choose $a_0$ and $b_0$ to guarantee convergence of the above algorithm when specializing to free addition.  



\begin{remark}
	In practice, we use Newton iteration with arbitrary initial guesses, which lacks guaranteed convergence.    While we only discussed the computation of $G_\mu^{-1}$, computing its derivative is straightforward since
	$$(G_\mu^{-1})'(y) = {1 \over G_\mu'(G_\mu^{-1}(y))}$$
and the  formul\ae\ for Cauchy transforms of admissible measures can be trivially differentiated. Similar logic allows us to compute the second derivative needed to perform the Newton iteration.  
\end{remark}

\section{Free additive convolution}

We now specialize the algorithm of the preceding section to free addition.  
To guarantee  convergence of the  algorithm, we must  accomplish two tasks:  generate a point set $\mathbf y_M$ so that \algref{choosepoints} succeeds and choose $(a_0,b_0)$ so that \algref{sqrtsupp} is guaranteed to converge.


Under the hypotheses of Theorems~\ref{sqrtgamma} and \ref{smoothgamma}, if we can construct ${\mathbf y}_M$ that lie in $G_{\mu_A}(\C^-) \cap G_{\mu_B}(\C^-)$, then \algref{choosepoints} will successfully select the subset of points lying inside $G_{\mu_A \bxp \mu_B}(\C)$.    We generate one such set of points as follows:

\Algorithm{pointsetgen} Generate point clouds
	\given  a smoothly or square root decaying measure $\mu$;
	\return a set of points $\mathbf y_M$ lying in $G_{\mu_A}(\C^-) \cap G_{\mu_B}(\C^-)$;
	\step Generate a point cloud $\mathbf d_M$ on the unit disk by taking a tensor product of $\mathbf u_m$ with  $M_{(0,1)}^{-1}(\mathbf x_m)$, the $m$ Chebyshev points on $(0,1)$;
	\step If $\supp \mu$ is the real line, generate a set of points lying in the lower half plane by 
	$$\mathbf z_{\mu,M} = \{z \in i {1 - \mathbf d_M \over 1 + \mathbf d_M} : \Im z < 0\};$$ 
otherwise, if $\supp \mu$ is an interval $(a,b)$, generate a set of points lying off $\supp \mu$ in the lower half plane by
	$$\mathbf z_{\mu,M} = \{z \in M_{(a,b)}(J(\mathbf d_M)) : \Im z < 0\};$$ 
	\step Define
	 $$\mathbf y_M = \{y \in G_{\mu_A}(\mathbf z_{\mu_A,M}) : G_{\mu_B}(G_{\mu_B}^{-1}(y)) = y\}.$$ 

We can now prove convergence of the full algorithm for the Schwartz class case:
\begin{Th}\label{smoothconvtotrue}
	Suppose that $\mu_A$ and $\mu_B$ satisfy the hypotheses of Theorem~\ref{smoothgamma}.  Then the output of \algref{computemeasure} converges in mapped ${\rm L}^2$ to $\mu_C = \mu_A \bxp \mu_B$ with
	$$G_{\mu_C}^{-1}(y) = G_{\mu_A}^{-1}(y) + G_{\mu_B}^{-1}(y) - {1 \over y}$$
when the assumed form of the measure is smoothly decaying and  ${\mathbf y}_M$ is computed via \algref{pointsetgen},  provided that $M \rightarrow \infty$ grows  sufficiently fast with $n \rightarrow \infty$.  
\end{Th}
\begin{proof}
	Theorem~\ref{smoothgamma} ensures that $\mu_C = \mu_A \bxp \mu_B$ is an invertible Schwartz measure, and that ${\mathbf y}_M$ lies inside $G_{\mu_A}(\C^-) \cap G_{\mu_B}(\C^-)$.  The analyticity of the operations in \algref{pointsetgen} ensure that ${\mathbf y}_M$ has a ``nice'' density, thus the hypotheses of Lemma~\ref{vandermondeconvergence} and Theorem~\ref{methconv} are satisfied.  
\end{proof}

We now choose $a_0$ and $b_0$ so that \algref{sqrtsupp} converges:

\begin{propo}
	Suppose $\mu_C = \mu_A \bxp \mu_B$ where $\mu_A$ and $\mu_B$ satisfy the hypotheses of Theorem~\ref{sqrtgamma}.  \algref{sqrtsupp} will converge  to $\supp \mu$ with the choice
	 $$a_0 = \max(G_{\mu_A}(\min\supp \mu_A),G_{\mu_B}(\min\supp \mu_B))$$
 and
	$$b_0 = \min(G_{\mu_A}(\max\supp \mu_A),G_{\mu_B}(\max\supp \mu_B)).$$
\end{propo}

\begin{proof}
	From the proof of Theorem~\ref{sqrtgamma}, we know that $g'$ only vanishes at $\xi_a$ and $\xi_b$ between $(a_0,b_0)$; thus, convergence of bisection follows.  

\end{proof}

\begin{Th}\label{sqrtconvtotrue}
	Suppose that $\mu_A$ and $\mu_B$ satisfy the hypotheses of Theorem~\ref{sqrtgamma}.  Then the output of \algref{computemeasure} converges in mapped ${\rm L}^2$ to $\mu_C = \mu_B \bxp \mu_C$ with
	$$G_{\mu_C}^{-1}(y)  = G_{\mu_A}^{-1}(y) + G_{\mu_B}^{-1}(y) - {1 \over y}$$
when the assumed form of the measure is square root decaying,   ${\mathbf y}_M$ is computed by \algref{pointsetgen} and $(a_0,b_0)$ are defined as above, provided that  $M \rightarrow \infty$ grows  sufficiently fast with $n \rightarrow \infty$.  
\end{Th}
\begin{proof}
	Theorem~\ref{sqrtgamma} ensures that $\mu_C = \mu_A \bxp \mu_B$ is an invertible square root decaying measure, and that ${\mathbf y}_M$ lies inside $G_{\mu_A}(\C^-) \cap G_{\mu_B}(\C^-)$.  The preceding proposition ensures that $\supp \mu_C$ is calculated via \algref{sqrtsupp}.  The analyticity of the operations in \algref{pointsetgen} ensure that ${\mathbf y}_M$ has a ``nice'' density, thus the hypotheses of Lemma~\ref{vandermondeconvergence} and Theorem~\ref{sqrtconv} are satisfied.  
\end{proof}

%

%

\subsection{Numerical examples}

\begin{remark}
	Throughout the paper, we use mean zero and  variance ${1\over \sqrt 2}$ for  Gaussian distributions unless otherwise specified.   $S_n$ denotes an $n \times n$ random symmetric matrix, constructed by  generating a random matrix $A_n$ with Gaussian distributed entries and defining
	$$S_n = {A_n + A_n^\top \over \sqrt{2 n}}.$$
$Q_n$ denotes a random orthogonal matrix, generated by computing the QR decomposition of $A_n$.   Finally, we generate a histogram associated with a random matrix ensemble $B_n$ by computing the eigenvalues of 100 instances of $B_n$.
\end{remark}

In Figure \ref{SemiPlusGauss}, we plot the numerically calculated free addition $\mu_G \bxp \mu_S$ of a Gaussian distribution $\mu_G$ with a semicircle distribution $\mu_S$.  This distribution was shown in \cite{bryc2006spectral} to be the limiting eigenvalue distribution of a class of Markov matrices. The left graph contains a plot of a Gaussian distribution (dotted), semicircle distribution (dashed) and their free addition (plain).  The right graph compares the computed free addition with a histogram of $Q_{150} \Lambda_{150} Q_{150}^\top + S_{150}$, where $\Lambda_n$ is a $n \times n$ diagonal matrix whose entries are Gaussian distributed.

	\begin{figure}[tb]
  \begin{center}
	\includegraphics[width=.4\linewidth]{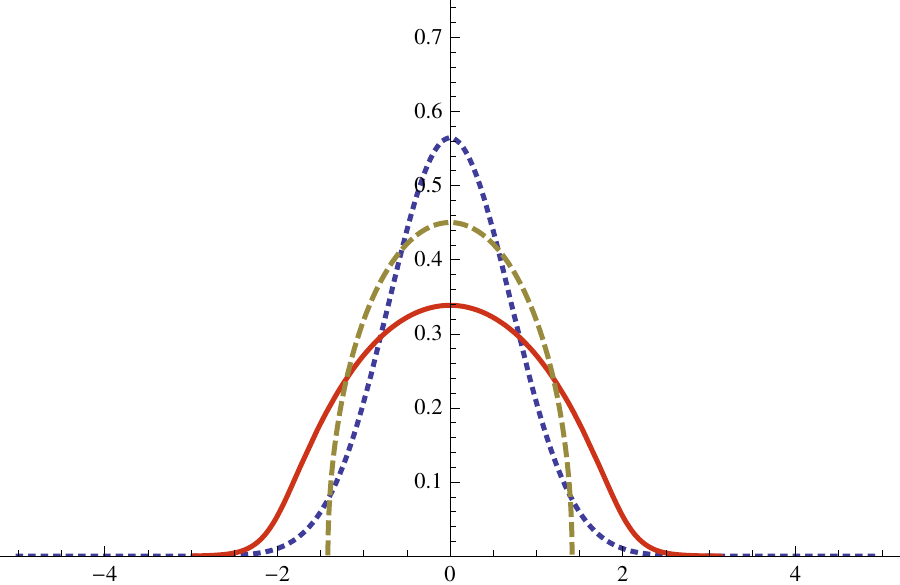}	 \includegraphics[width=.4\linewidth]{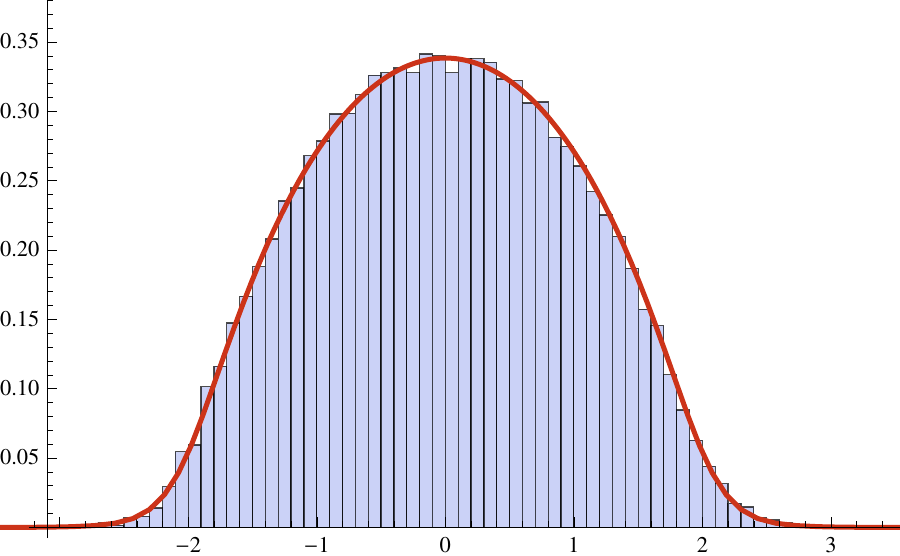}
  \end{center}
  \caption{Free addition of a Gaussian distribution with a semicircle distribution.  }
\label{SemiPlusGauss}
\end{figure}

	\begin{figure}[tb]
  \begin{center}
	\includegraphics[width=.4\linewidth]{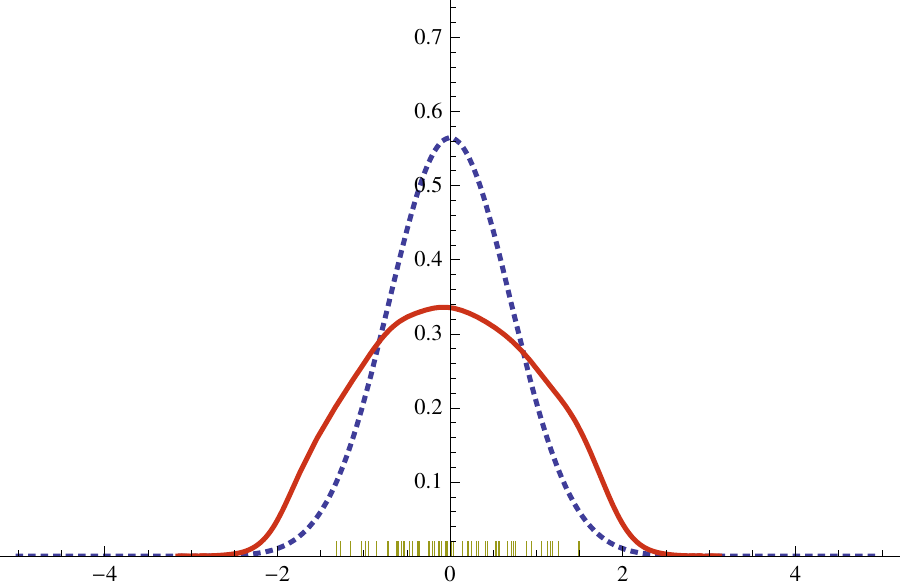}	 \includegraphics[width=.4\linewidth]{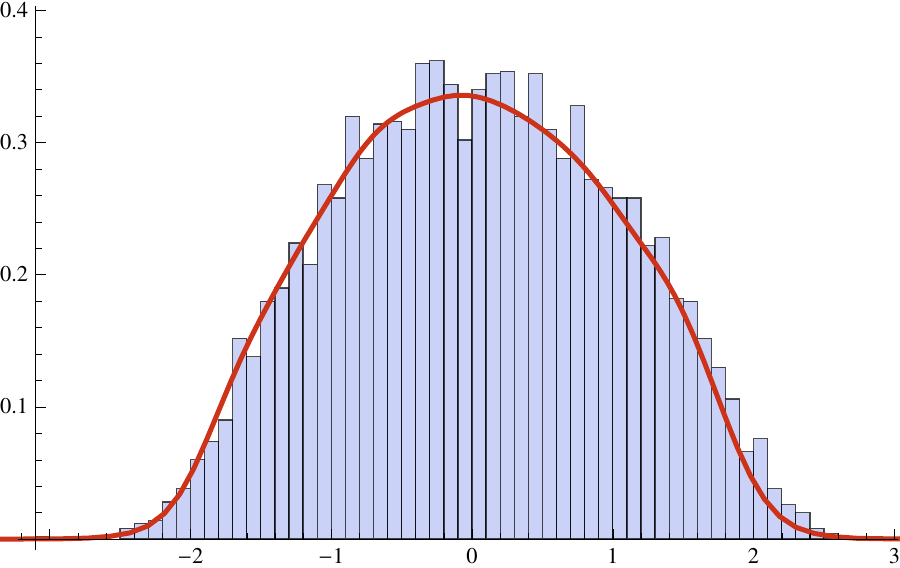}
  \end{center}
  \caption{Free addition of a Gaussian distribution with a single instance of an approximate semicircle distribution.   }
\label{DSemiPlusGauss}
\end{figure}

Often one does not have exact expressions for the limiting distributions of the eigenvalues, but rather, one can sample a single instance from the distribution.  In this case, the counting measure --- a sum of point measures --- over this single instance can be calculated.  In Figure \ref{DSemiPlusGauss} we repeat the experiment of Figure \ref{SemiPlusGauss} where the semicircle distribution is replaced with the counting measure $\mu_{A_{50}}$ of a single matrix $A_{50}$  drawn from $S_{50}$.  On the right, we compare the computed distribution with the histogram of $Q_{50} \Lambda_{50} Q_{50}^\top + A_{50}$, where $A_{50}$ is now a {\it fixed} matrix.

	\begin{figure}[tb]
  \begin{center}
	\includegraphics[width=.4\linewidth]{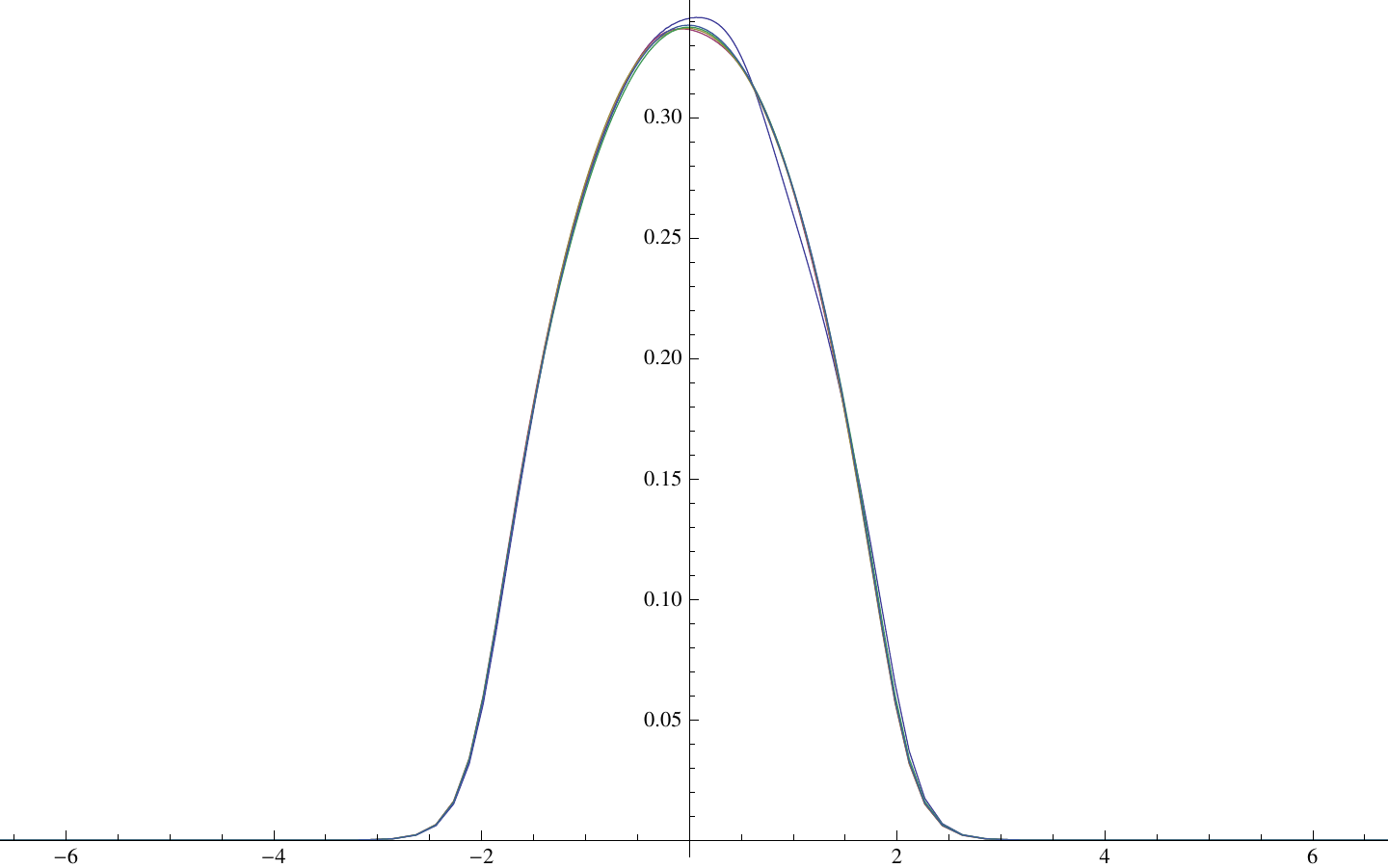}	 \includegraphics[width=.4\linewidth]{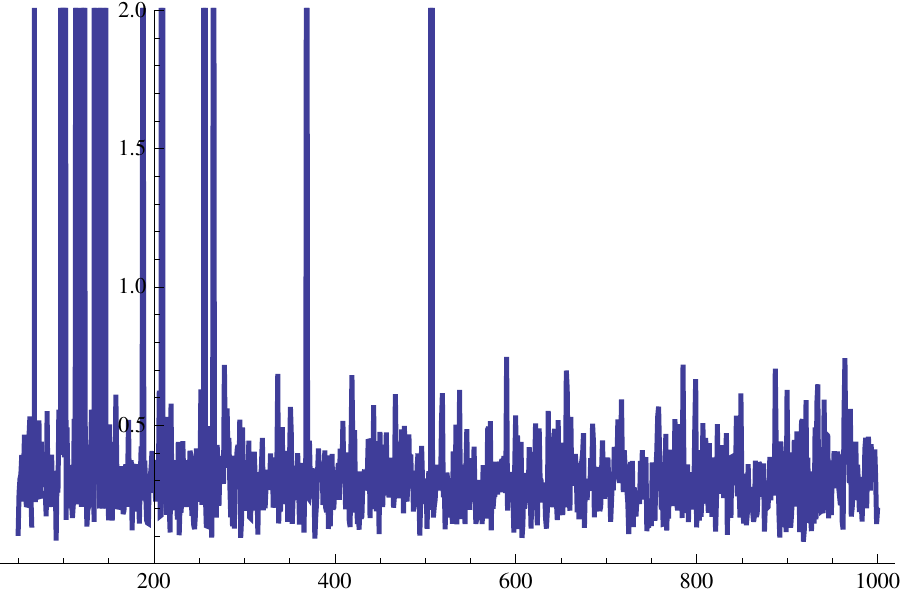}
  \end{center}
  \caption{Free addition of a Gaussian distribution with approximate semicircle distributions for $n = 100,200,\ldots,400$ (left).  The scaled (by $n$) Kolmogorov--Smirnov distance ($D_n = \sup x | F_n(x) - F(x) |$ where $F(x)$ is the distribution in Figure \ref{SemiPlusGauss} and $F_n$ is the distribution in \ref{DSemiPlusGauss}) between the cdfs illustrating convergence in the respective cumulative distribution functions (right). }
\label{DSemiPlusGaussConv}
\end{figure}

	As $n \rightarrow \infty$,  $\mu_{S_n} \bxp \mu_G$ will converge in some sense to $\mu_S \bxp \mu_G$, as seen in the right hand of Figure \ref{DSemiPlusGaussConv}.  We can estimate this growth by comparing the maximum difference of the cdf of computed measures for growing values of $n$.  In the right-hand side of Figure \ref{DSemiPlusGaussConv}, we plot this scaled by $n$, demonstrating that the convergence rate appears to be $O(n^{-1})$.


%
%


	\begin{figure}[tb]
  \begin{center}
	 \includegraphics[width=.4\linewidth]{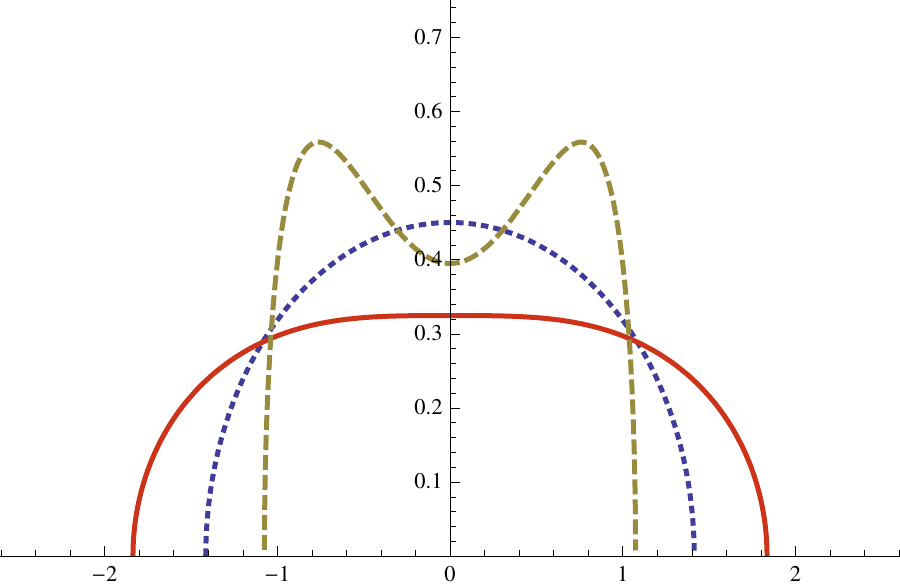}\includegraphics[width=.4\linewidth]{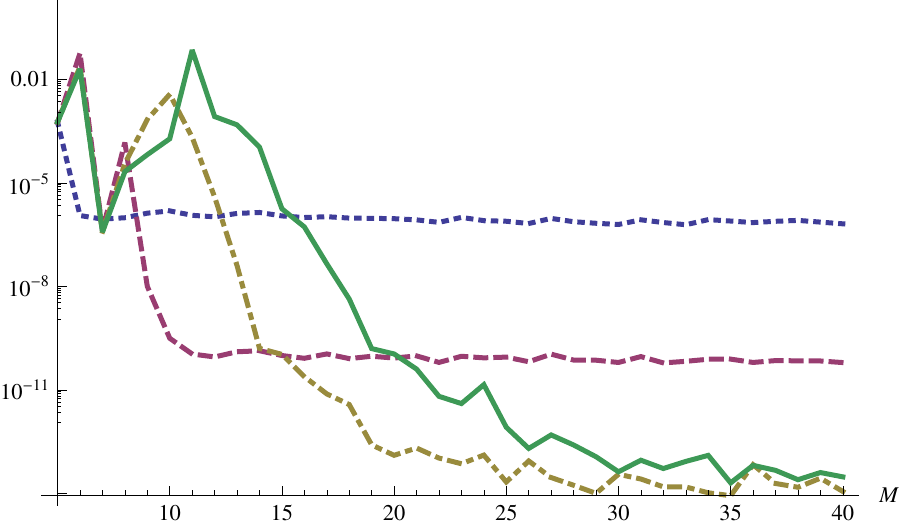}
  \end{center}
  \caption{Free addition of a Semicircle distribution with the equilibrium measure associated with the potential $V(x) = x^4$.  The pointwise error of compared to the exact solution for  $n = 20$ (dotted), 40 (dashed), 60 (dash--dotted) and 80 (plain).}
\label{SemiPlusQuartic}
\end{figure}

In Figure \ref{SemiPlusQuartic} we  compute a measure which is square root decaying.  Here we define $\mu_{4}$ as the equilibrium measure of the potential $V(x) = x^4$ (see \cite{st96} for definition of equilibrium measures), which we know in closed form \cite{d00}.  We then calculate $\mu_S \bxp \mu_{4}$ using \algref{computemeasure}.   There is no obvious way of generating a histogram for this measure; hence, unlike other examples, there is no known Monte Carlo approach for approximating $\mu_S \bxp \mu_{4}$.  However, this is an example which was calculated symbolically in \cite{re08}, hence we can compare our numerically computed measure with the exact measure.  We plot the error for  $n = 20$ (dotted), 40 (dashed), 60 (dash--dotted) and 80 (plain) as $M$ increases. Recall that $n$ is the number of coefficients in the Chebyshev representation of $\mu_S \bxp \mu_4$ while $M$ is the number of points in the point cloud used in the least-squares based measure recovery algorithm described in Algorithm 6. The error is computed by taking the maximum error over 100 Chebyshev points on the interval $\supp (\mu_S \bxp \mu_{4})$.

	\begin{figure}[tb]
  \begin{center}
	\includegraphics[width=.4\linewidth]{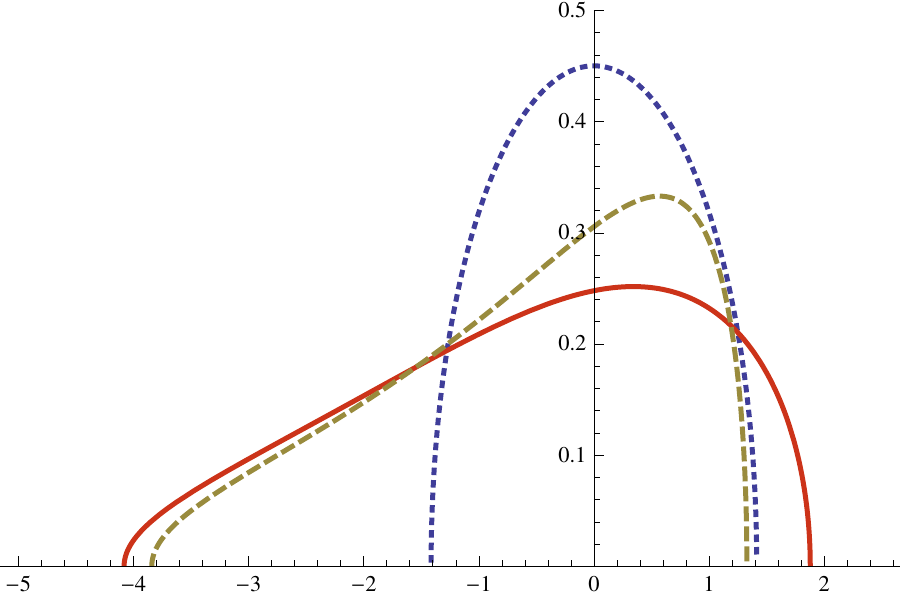}
  \end{center}
  \caption{Free addition of a Semicircle distribution with the equilibrium measure associated with the potential $V(x) = e^x - x$.  }
\label{SemiPlusEM}
\end{figure}

In Figure \ref{SemiPlusEM}, we define $\mu_{EM}$ as the equilibrium measure of the potential $V(x) = e^x - x$ --- which we calculate numerically (in the required form) using the approach of \cite{SOEquilibriumMeasure} ---  and  then calculate $\mu_S \bxp \mu_{EM}$.  This is an example which cannot be computed symbolically, at least using the framework of \cite{re08}.

	\begin{figure}[tb] 
  \begin{center}
	\includegraphics[width=.4\linewidth]{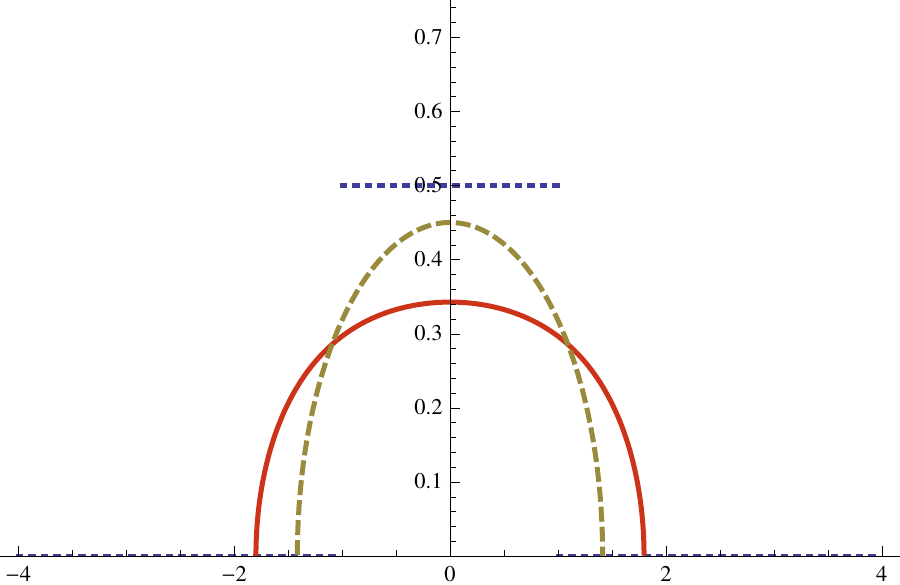}	 \includegraphics[width=.4\linewidth]{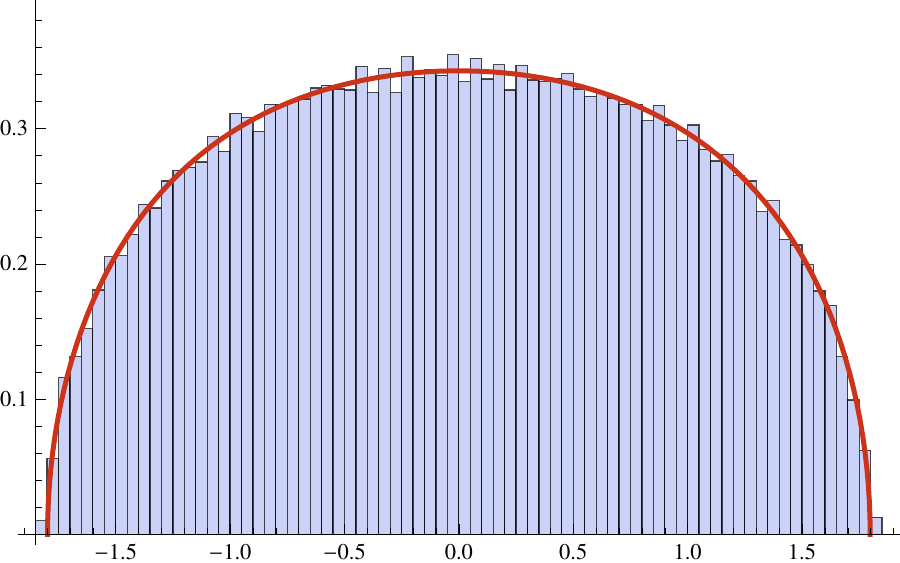}
  \end{center}
  \caption{Free addition of a step distribution with a semicircle distribution.   }
\label{SemiPlusSquare}
\end{figure}

%
Finally, in Figure \ref{SemiPlusSquare} we calculate the free addition of a semicircle distribution with a step distribution $\mu_S \bxp ( {1 \over 2} {\mathbf 1}_{(-1,1)})$, demonstrating that a square root decaying measure arises.  While ${1 \over 2}{\mathbf 1}_{(-1,1)}$ We do not need to use the ellipse method for this measure, as we can calculate its Cauchy transform and  inverse Cauchy transform explicitly:
	\begin{align*}
	G_{ {1 \over 2} {\mathbf 1}_{(-1,1)}}(z) & = {\log(1 + z) - \log(z-1) \over 2} \hbox{ and }\cr
	G_{ {1 \over 2} {\mathbf 1}_{(-1,1)}}^{-1}(y) & = {\coth {y \over 2} + \tanh{y \over 2} \over 2}.
	\end{align*}
  We compare the computed distribution with the histogram of
	$Q_{300} \Lambda_{300} Q_{300}^\top + S_{300},$
where $\Lambda_n$ is a diagonal matrix whose entries are evenly distribution on $(-1,1)$.

\section{Free multiplicative convolution and the S transform}

In the case where $\mu\neq \delta_0$ and the support of $\mu$ is contained in $[0, +\infty)$,
 one   also defines its {\it $T$-transform}
 $$T_\mu(z)=\int\f{x}{z-x}\ud \mu(x)  \qquad \textrm{for } z \notin \supp \mu.$$
%
The {\it $S$-transform}, defined as
$$S_\mu(y):=(1+y)/(y {T_\mu^{-1}(y)}),$$
is the analogue  of the Fourier transform for  free multiplicative convolution $\bxt$. The free multiplicative convolution of two \pro measures $\mu_A$ and $\mu_B$ is denoted by the symbols $\bxt$ and can be characterized as follows.

Let $A_n$ and $B_n$ be independent $n \times n$ symmetric (or Hermitian) positive-definite random matrices that   are invariant, in law, by conjugation by any orthogonal (or unitary) matrix. Suppose that, as $n \rightarrow \infty$, $\mu_{A_{n}} \rightharpoonup \mu_{A}$ and $\mu_{B_{n}} \rightharpoonup \mu_{B}$. Then, free probability theory states \cite{voiculescu1987multiplication} that $\mu_{A_n \cdot B_n} \rightharpoonup \mu_{A} \bxt \mu_{B}$, a \pro measure which can be characterized in terms of the $S$-transform as
$$S_{\mu_A\bxt\mu_B}(z)= S_{\mu_A}(z)S_{\mu_B}(z).$$

	The T transform can be computed in the same way as the Cauchy transform; we only need to multiply the representation of the measure by $x$ beforehand.  The numerical method for calculating the inverse Cauchy transform proceeds as before.  From the relationship of the S transform, we know that
	$$T_{\mu_A\bxt\mu_B}^{-1}(y) = T_{\mu_A}^{-1}(y)T_{\mu_B}^{-1}(y) {y \over 1 + y}.$$
Note that $T_{\mu_A \bxt \mu_B} = G_{\mu_C}$, for the (non-probability) measure $\mu_C$ defined by
	$$\ud \mu_C(x) = x \ud[\mu_A \bxt \mu_B](x).$$
Therefore, we can use the \algref{computemeasure} to find $\ud \mu_C$, and in turn $\mu_A \bxt \mu_B$.   Similar to free addition, we use the point cloud $\mathbf y_M =T_{\mu_A}(\mathbf z_{\mu_A,M})$.  While we omit the proof of convergence, it should follow along the same lines as Theorems~\ref{smoothconvtotrue} and \ref{sqrtconvtotrue}.
	
%

\subsection{Numerical examples}

	\begin{figure}[tb]
  \begin{center}
	\includegraphics[width=.4\linewidth]{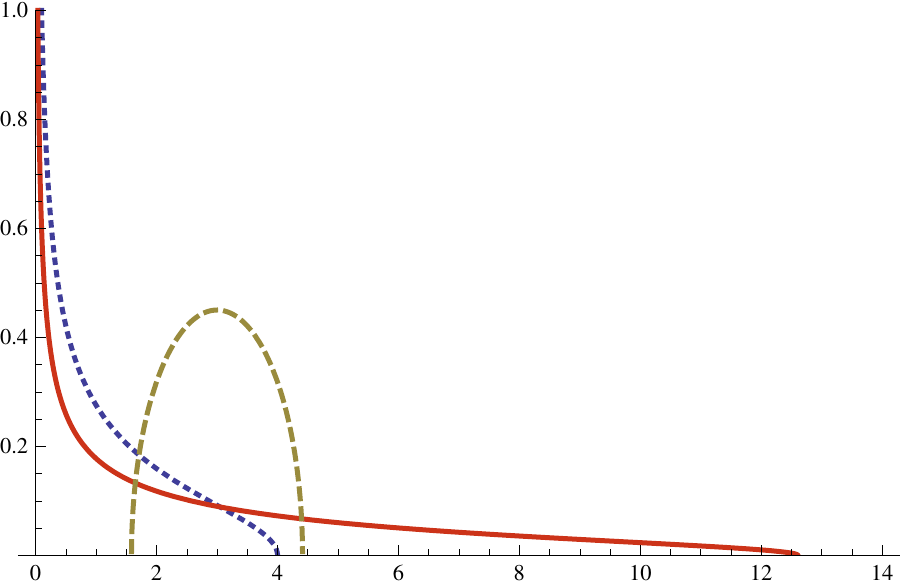}	 \includegraphics[width=.4\linewidth]{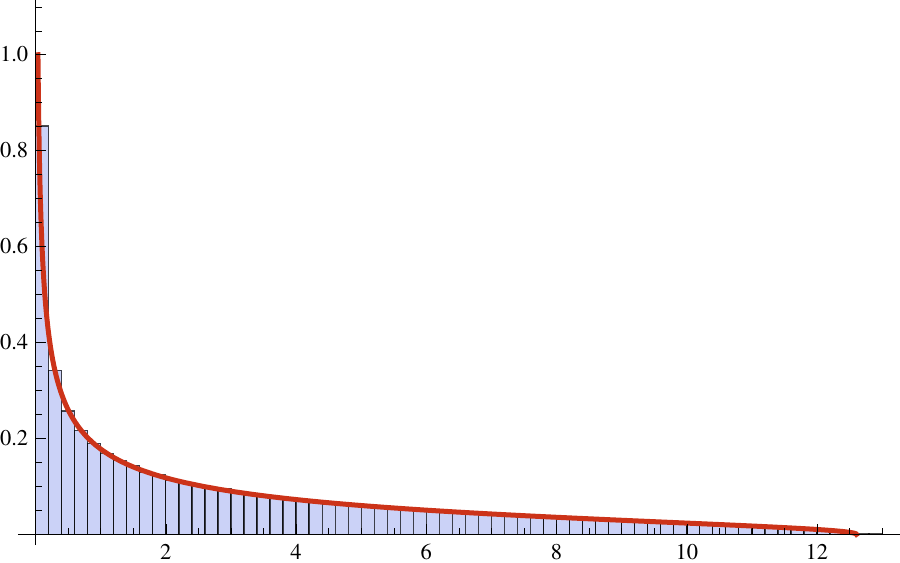}
  \end{center}
  \caption{Free times of a shifted semicircle distribution with a Mar\v{c}enko--Pastur distribution.   }
\label{ShiftedSemiPlusMP}
\end{figure}

In Figure \ref{ShiftedSemiPlusMP}, we consider the problem of computing a free product of a a shifted semicircle distribution with a singular Mar\v{c}enko--Pastur distribution
	$$\ud \mu_{MP}(x) = {\sqrt{4 - x} \over 2 \pi \sqrt{x}} \ud x.$$
  While this distribution is not admissible, it is when we multiply by $x$; as in the definition of the T-transform.  The procedure then works as before.  We compare the computed measure with a histogram of
  	$$B_{200} B_{200}^\top (S_{200} + 3 I),$$
  where $B_n = {1 \over \sqrt{n}} A_n$ and $A_n$ is an $n \times n$ random matrix with Gaussian distributed entries, now with mean zero and variance one.

\section{Free compression}

Let $B_n$ be the $n \times n$  matrix generated by  taking the upper left $n \times n$ block  of $Q_m A_m Q_m^\top$,  where $n \leq m$.   If the eigenvalues of $A_m$ tend to the distribution $\mu$, then the eigenvalues of $B_n$ tend to the {\it free compression} of $\mu$, i.e.,
	$$\mu_{B_n} \rightharpoonup {m \over n} \bxd \mu.$$

 Let $\alpha  \in (0,1]$.   We have that \cite{nica1996multiplication}
$$R_{\alpha \bxd \mu}(z) = R_{\mu}(\alpha z).$$
	Rearranging the definition of the R transform, we find that
	$$G_{\alpha \bxd \mu}^{-1}(y) = G_{\mu}^{-1}(\alpha y) + {1 \over y} - {1 \over \alpha y}.$$
Therefore, we can apply \algref{computemeasure} to compute $\alpha \bxd \mu$, with the point cloud $\mathbf y_M =G_{\mu}(\mathbf z_{\mu,M})$.  Again, we omit the proof of convergence.

\subsection{Numerical examples}

	In Figure \ref{GaussCompress}, motivated by the theoretical results in \cite{belinschi2010normal}, we compare the compute free compression of a Gaussian distribution with a histogram of the $\alpha 300 \times \alpha 300$ principal block of $Q_{300} \Lambda_{300} Q_{300}^\top$, where $\Lambda_n$ is an $n \times n$ diagonal matrix whose entries are Gaussian distributed.

	\begin{figure}[tb]
  \begin{center}
	\includegraphics[width=.9\linewidth]{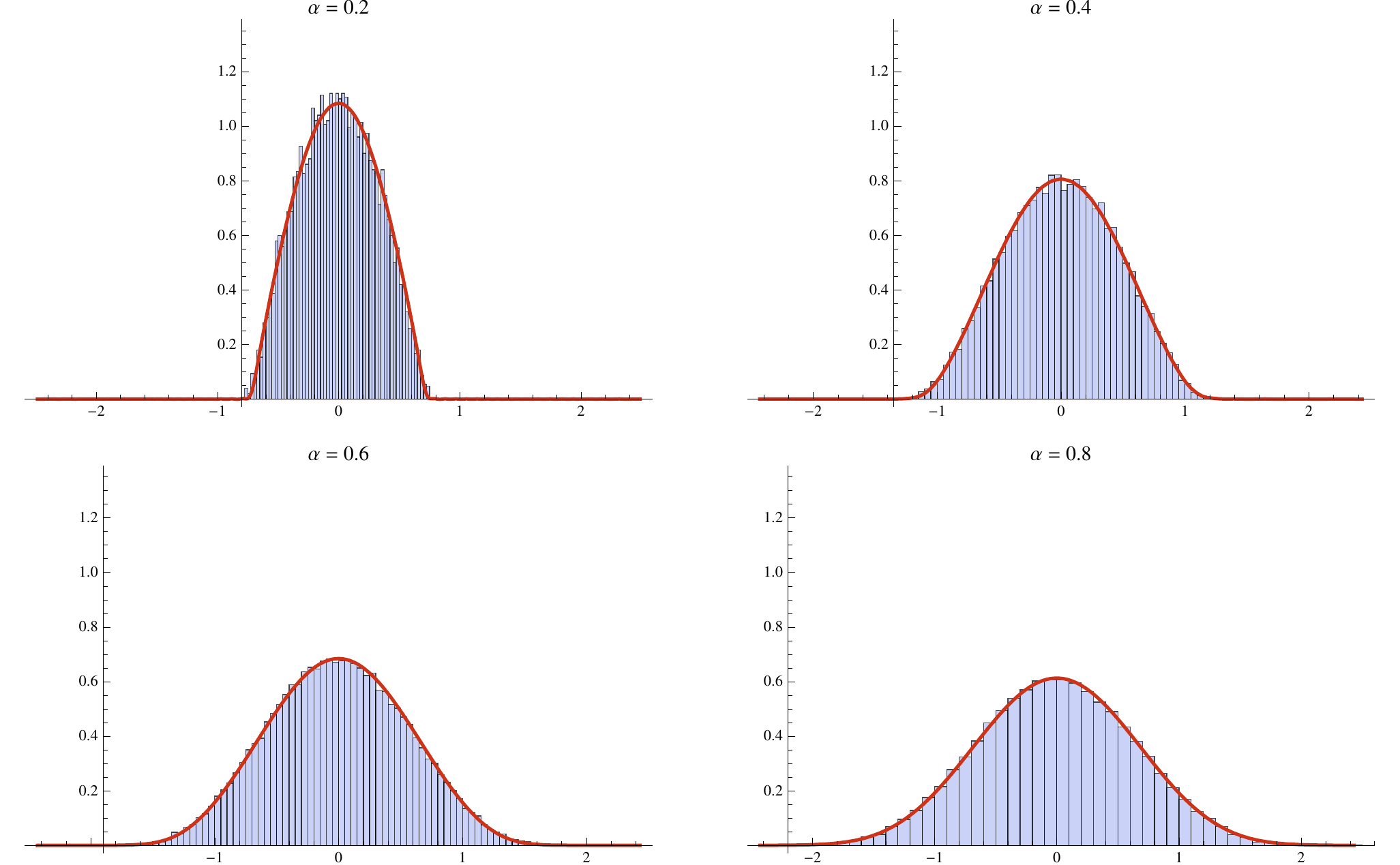}
  \end{center}
  \caption{Four compressions of a Gaussian. }
\label{GaussCompress}
\end{figure}

\section{Extensions}

We now identify some extensions of the proposed method:
\begin{itemize}
\item Numerical free convolution of measures supported on multiple intervals.  Here there are two issues that must be overcome: computation of the inverse Cauchy transform, and determination of the support of the measure.    The major complication is that the inverse Cauchy transform is multi-valued.
\item Free rectangular convolution (see \cite{benaych2009rectangular}).  This operation inherently requires computation with measures supported on multiple intervals.
\end{itemize}
A preliminary software implementation is available in {\sc Mathematica} as a component of {\sc RHPackage} \cite{RHPackage}.

\section*{Acknowledgements}
We thank Serban Belinschi for many insightful comments regarding the regularity properties of free convolution.  We thank Ben Adcock for suggesting the proof of  Lemma \ref{vandermondeconvergence}.  We thank Folkmar Bornemann for initiating this collaboration by pointing R.R.N to S.O's work \cite{SOHilbertTransform} in response to a query about whether free convolutions might be computable numerically. R.R.N's work was supported by an Office of Naval Research Young Investigator Award N00014-11-1-0660, ARO MURI W911NF-11-1-0391 and NSF award CCF-1116115.  We thank the anonymous referee for her or his helpful criticisms.  

\appendix

\section{Properties of the Cauchy transform and its inverse for invertible measures}\label{cauchyprops}

We describe properties of the special class of invertible measures that we use to justify the arguments below.  The proofs are by no means novel, however, we include them here as we are unaware of a convenient reference.  

\begin{propo} \label{cauchyconj}
	The Cauchy transform commutes with conjugation:
	$$G_\mu(\bar z) = \overline{G_\mu(z)}.$$

\end{propo}
\begin{proof}
	Follows since our measures are real and $$\overline{1 \over z - x } = {1 \over  \bar z - x}.$$
\end{proof}

\begin{propo} \label{realsign}
	The Cauchy transform satisfies
	$$\Im G_\mu(z) < 0 \hbox{ for } \Im z > 0$$
and 
	$$\Im G_\mu(z) > 0 \hbox{ for } \Im z < 0.$$
If $\mu$ is smoothly or square root decaying, then
	$$\Im G_\mu^-(z) > 0\hbox{ and }\Im G_\mu^+(z) < 0$$
for $z \in \supp \mu$.  

\end{propo}
\begin{proof}
	For $\Im z > 0$, we have
	$$\Im \int {d \mu \over z - x} = \int \Im\left[{1 \over z - x} \right] d \mu = - \Im z \int  {  d \mu  \over |z- x|^2} < 0.$$
The second part of the theorem follows since invertible measures can be written as $d \mu = \psi(x) d x$ where $\psi$ is H\"older continuous, hence Plemelj's lemma holds:
	$$G_\mu^\pm(z) = \mp \pi i \psi(z) + \dashint {d \mu(x) \over z-x}.$$

\end{proof}

\begin{propo}\label{monotonic}
	Suppose $\mu$ has compact support inside the interval $(a,b)$.  The Cauchy transform is monotonically decreasing for $z > b$ and $z < a$.  Assuming the Cauchy transform is single-valued, then $G_\mu^{-1}$ is  monotonically decreasing inside $(G_\mu(a), G_\mu(b))$.  

\end{propo}

\begin{proof}
	Follows from differentiation:
	$${d \over dz} \int {d \mu(x) \over x - z} =\int {d \mu(x) \over (x - z)^2} > 0$$
for real $z$.

\end{proof}

\begin{propo}\label{dominance}
	Suppose $\mu$ is invertible.  Then
	$$\Im {1 \over G_\mu(z)} < \Im z$$
for $z \in \supp \mu$ and $\Im z < 0$.  Hence,
	$$ \Im {1 \over w} < \Im G_{\mu}^{-1}(w)$$
for $w \in G_\mu^-(\supp \mu)$ and $\Im w > 0$.  
\end{propo}

\def\R{{\mathbb R}}

\begin{proof}
	$$\Im\left({1 \over G_\mu(z)}  - z\right)$$
is harmonic.  On the real axis, we have
	$$\Im\left({1 \over G_\mu^-(z)} -z \right) = \Im {1 \over G_\mu^-(z)} < 0$$
for $z \in \supp \mu$, while it is zero elsewhere.  From the modulus maximization property of harmonic functions, it must be the case that
	$$\Im\left({1 \over G_\mu(z)}  - z\right) < 0$$
for $\Im z < 0$.  Letting $z = G_\mu^{-1}(w)$ shows the second result.  

\end{proof}

%
%
%
%
%

\begin{propo}\label{derjensen}
	Suppose $\mu$ is invertible and has compact support.  Then
	$$  G_\mu'(z) + G_\mu(z)^2 < 0$$
for $z \geq \max \supp \mu$.  Hence,
	$${G_{\mu}^{-1}}'(w) + {1 \over w^2} > 0$$
for $w \in (0,G_\mu^\pm(\max \supp \mu)]$.
\end{propo}

\begin{proof}

From Jensen's inequality, we have
	$$G_\mu(z)^2 = \left[\int {d \mu \over z - x}\right]^2 < \int {d \mu \over (z - x)^2} = - G_\mu'(z).$$
The second inequality follows from substituting $z = G_\mu^{-1}(w)$ and 
	$${G_\mu^{-1}}'(w) = {1 \over G'(G_\mu^{-1}(w))}.$$

\end{proof}

\begin{propo}\label{sqrtturning}
 
	If $\mu$ is a Jacobi measure and $\beta > 0$, then 
	$$G_\mu(z) \sim G_\mu(b) + C (z - b)^\beta + o(z - b)^\beta.$$
If $\mu$ is precisely a Jacobi measure and $\beta > 0$, then $C \neq 0$ and
	$$G_\mu^{-1}(w) \sim  b + C (G_\mu(b) - w)^{1/\beta} + o(G_\mu(b) - w)^{1/\beta}.$$
If $\mu$ is precisely square root decaying then $G_{\mu}^{-1}$ is analytic at $G_\mu(b)$, with a quadratic turning point.  

If $\mu$ is precisely a Jacobi measure and $\beta \leq 0$, then 
	$$G_\mu(b) = \infty.$$

Similar properties hold near $G_\mu(a)$.  

\end{propo}

\begin{proof}
	The case $\beta \leq 0$ follows since $(b - x)^{\beta - 1}$ is not integrable.  In the case $\beta > 0$, we subdivide $(a,b)$ into $(a,b_0)$ and $(b_0,b)$.  In the latter interval,  we write the Cauchy transform as
	\begin{align*}
		\int_{b_0}^b {d \mu \over z - x} &= \int_{b_0}^b {\psi(x) (x-a)^\alpha - \psi(b)  (b-a)^\alpha    \over z - x} (b-x)^\beta  dx \cr
		&\qquad\qquad+ \psi(b) (b-a)^\alpha \int_{b_0}^b {(b - x)^\beta dx \over z - x}.
		\end{align*}
From Plemelj's lemma, it follows that
	$$\int_{b_0}^b {(b - x)^\beta dx \over z - x} \sim (z - b)^\beta + \hbox{analytic}.$$
To see this for $\beta$ not an integer, represent the Cauchy transform as ${(z - b)^\beta \over \E^{\I \beta \pi} - \E^{-\I \beta \pi}}  + \hbox{analytic}$ in a circle surrounding $b$.  This satisfies the right jump, since
	\begin{align*}
		(z-b)_+^\beta - (z-b)_-^\beta & = \E^{\beta \log_+ (z-b)}  - \E^{\beta \log_-(z-b)} \cr
		&= \E^{\beta \log (b-x) + \I \beta \pi} -  \E^{\beta \log (b-x) - \I \beta \pi} = (b-x)^\beta (\E^{\I \beta \pi} - \E^{-\I \beta \pi}).
	\end{align*}
The rest of the theorem now follows.  
%
	
\end{proof}

\begin{propo}\label{smoothdecaydecay}
	Let $\mu$ be a Schwartz measure.  Then
	$$G_\mu(z) = {1 \over z} + {E[x] \over z^2} + {E[x^2] \over z^3} + \cdots$$
as $z \rightarrow \infty$, where 
	$$E[x^k] = \int x^k d\mu$$
are the moments.  Therefore, $G_\mu^{-1}$ has a real asymptotic expansion at zero:
	$$G_\mu^{-1}(y) \sim {1 \over y} + c_0 + c_1 z+ \cdots$$
as $y \rightarrow 0$ for $y \in G_\mu(\C)$ and $c_k$ real.

\end{propo}

\begin{proof}
The first part follows by replacing the Cauchy kernel with it's geometric series:
	$$G_\mu(z) = \int {d \mu \over z - x} =\sum_{k = 0}^n {1 \over z^{k+1}} \int x^k d \mu + {1 \over z^{n+1}} \int {x^{n+1} \over x-z} d\mu$$
The second part follows from inversion of asymptotic expansions.  
\end{proof}

\begin{propo}\label{continuousboundary}
	Suppose $\mu$ is an admissible measure.  Then 
	$G_\mu(\supp \mu)$
is a (possibly unbounded) continuous curve.   The curve is bounded if $\mu$ is a square root or smoothly decaying measure.

\end{propo}

\begin{proof}
	For square root and smoothly decaying measures, uniform convergence  of the series representations  of the Cauchy transforms (see Table~\ref{tab:rep of measures})  implies continuity and boundedness.  The proposition is trivial for point measures.

	 For Jacobi measures, without loss of generality we can assume $\psi(a) \neq 0$ and $\psi(b) \neq 0$.  We have that 
	$$G_\mu(z) = \int {(x - a)^\alpha (b - x)^\beta \psi(x) \over z - x} dx$$
is continuous, and for $\beta \leq 0$ it blows up as $ z \rightarrow b$ from the right, otherwise, it approaches a limit.  For $a < z < b$ we have from Plemelj's lemma:
	\begin{align*}
		G_{\mu}^\pm(z) &= \mp \I \pi {\psi(z)} +  \dashint {(x - a)^\alpha (b - x)^\beta \psi(x) \over z - x} dx \\
					&=\mp \I \pi {\psi(z)} +  \int {(x - a)^\alpha (b - x)^\beta (\psi(x) - \psi(z)) \over z - x} dx  + \psi(z) \dashint {(x - a)^\alpha (b - x)^\beta \over z - x} dx
	\end{align*}
The differentiability of  $\psi$  ensures the  continuity of the first integral.  The latter principal value integral can be expressed in closed form \cite{elliott71}, and for $\beta \leq 0$ it blows up as $z \rightarrow b$ from the left, otherwise, it approaches the same limit as from the right.   Similar logic proves continuity near $a$.


\end{proof}

\section{Convergence of Vandermonde systems with large number of points}

	The following proofs are straightforward (we thank Ben Adcock for help proving them), though we have not found them in precisely this form in the literature.  In this section, the norm is always  ${\rm L}^2$ (on the unit circle), $\ell^2$ or the matrix norm induced by $\ell^2$.  

\begin{propo}
	Let $\mathbf d_m = (d_1,\ldots,d_m)$ be a point cloud that covers the unit disk as $m \rightarrow \infty$, and
	 $$V = \begin{pmatrix}
	 		 1 & d_1 & \cdots & d_1^{n-1} \cr
			\vdots & \vdots & \ddots & \vdots \cr
			1 & d_m & \cdots & d_m^{n-1} \cr
		\end{pmatrix},$$
the $m \times n$ Vandermonde matrix associated with the point cloud.  Then for any  $n$, there exists $m$ large enough so that
	$\|V^+\| \leq  \sqrt{n} + \delta$, where $V^+$ denotes the Moore--Penrose pseudoinverse of $V$.
Furthermore, if, for all $|z|\leq 1$ and $\epsilon > 0$, the smallest $m$ such that
	$\min(|z - \mathbf d_m|) \leq \epsilon$
satisfies ${m} = O(\epsilon^{-\alpha})$ for some $\alpha > 0$,  then $m = O(n^{\alpha})$.
\end{propo}
\begin{proof}
	Assuming that $V^+$ has full column rank (which will follow from the argument below for large $m$),
	$$\| V^+ \| =  {1 \over \sigma_{\rm min}} = {1 \over  \inf_{\mathbf c \in \C^n, \|c\| = 1}  \|V {\mathbf c} \|}.$$
where $\sigma_{\rm min}$ is the smallest singular value.    For $m$ large enough, there exist $n$ points within ${1 \over n}$ of $n$ evenly spaced points $\mathbf u_n$.  (Under the secondary hypothesis, this $m$ clearly grows like $O(n^{\alpha})$.) Let $ V_g$ be the $n \times n$ Vandermonde matrix associated with these points, so that (under a certain ordering)
	$$V = \begin{pmatrix} V_g \cr V_b\end{pmatrix}.$$
 Then
	$$\| V \mathbf c \| = \left\| \begin{pmatrix}  V_g \mathbf c \cr V_b \mathbf c  \end{pmatrix}  \right\| \geq \| V_g \mathbf c \|.$$

We have
	$$V_g = V_u + {1 \over n} { \Delta},$$
where $V_u$ is the Vandermonde matrix associated with $\mathbf u_n$ (i.e., a discrete Fourier transform) and $\|\Delta\| \leq 1$.  Thus
 $\| V_g \mathbf c\| = \| V_u \mathbf c\| + O({1 \over n})$.  We know  that
	$$\inf_{\mathbf c \in \C^n, \|c\| = 1} \| V_u \mathbf c\| = {1 \over \| V_u^{-1} \|} = {1 \over \sqrt{n}}$$
which completes the proof.

\end{proof}

\begin{lem}\label{vandermondeconvergence}
	Suppose that, for all $|z|\leq 1$ and $\epsilon > 0$, the smallest $m$ such that
	$\min(|z - \mathbf d_m|) \leq \epsilon$
satisfies ${ m} = O(\epsilon^{-\alpha})$ for some $\alpha > 0$.
If $f$ is analytic in the unit disk, then for $m$ large enough the least squares approximation of $f$ at the points $\mathbf d_m$ converges to $f$ in ${\rm L}^2$.

\end{lem}

\begin{proof}
	Let
 	$$P_n=  (I_n, {\mathbf 0})$$ denote the $n \times \infty$ projection operator and let $E_{m}$ be the $m \times \infty$ operator defined by
	$$E_{m} f = f(\mathbf d_m).$$
Then we are approximating $f$ by
	 $$\tilde f = P_n^\top V^+ E_{m} f.$$
Furthermore,
	$$P_{n} f=  V^+ E_{m} P_n^\top P_n f.$$

	 We thus have the error
	\begin{align*}
		f - \tilde f =  f -  P_n^\top P_n f + P_n^\top V^{+} E_{m} ( P_n^\top P_n f - f) \cr
		 = (I -   P_n^\top V^{+} E_{m}) (f -  P_n^\top P_n f ).
		\end{align*}
In other words,
	$$\|f - \tilde f\| \leq (1+ m(n^{1/2} + \epsilon)) \|f - P_n^\top P_n f\|.$$
$\|f - P_n^\top P_n f\|$ decays exponentially fast for any analytic $f$.  The theorem follows since $m$  grows at most algebraically with $n$.

\end{proof}

We need to modify the preceding lemma for the least squares system used in \algref{smoothlydecaying}, which is not quite Vandermonde:

\begin{cor}\label{vandermondeconvergencevanish}
	Suppose $f$ is analytic inside the unit disk, smooth on the boundary and satisfies
	$$f(-1) = 0.$$
Then for $m$ large enough the least squares approximation
	$$\sum_{k=1}^m \psi_k (d_j^k - (-1)^{k}) \approx f(d_j)$$
converges to $f$ in ${\rm L}^2$.  
\end{cor}

\begin{proof}
The logic of the preceding proofs still follow. To see this, define the $n$ points $ \mathbf{\tilde u}_n$ as the points $\mathbf u_{n+1}$ with the point $-1$ removed.  Interpolating $f$ at $ \mathbf{\tilde u}_n$ by $(z+1,z^2-1,\ldots,z^n-(-1)^n)$ will also interpolate $f$ at $-1$, hence it is equivalent to interpolating $f$ at $\mathbf u_{n+1}$.  Thus the norm of the inverse of the relevant interpolation matrix at $\mathbf u_n$ is  bounded by $\sqrt{n + 1}$.  The truncation error
	$$\|{f - P_n^\top P_n f}\|$$
now decays only super-algebraically fast, but that is sufficient for convergence.  

\end{proof}

\end{document}